\documentclass[a4paper,10pt]{article}

\usepackage{graphics}
\usepackage{color}
\usepackage{xcolor}
\usepackage{comment}
\usepackage{amsmath,amsthm,mathtools}
\usepackage{amsfonts}
\usepackage{mathrsfs}
\usepackage{color}
\usepackage{algorithm}
\usepackage{algorithmic}
\usepackage{siunitx}
\usepackage{multirow}
\usepackage{stackrel}
\usepackage{booktabs}
\usepackage{tabularx}
\usepackage{arydshln}
\usepackage{cite}
\usepackage{url}
\usepackage[demo]{graphicx}
\usepackage{listings}

\theoremstyle{definition}
\numberwithin{equation}{section}
\newtheorem{Theorem}{Theorem}
\newtheorem{Example}{Example}

\newtheorem{Lemma}{Lemma}

\newtheorem{Remark}{Remark}

\newcommand{\Qa}{\text{$Q_{3,+}$}}
\newcommand{\Qb}{\text{$Q_{4,+}$}}
\newcommand{\Qainv}{\text{$Q_{3,+}^{-1}$}}
\newcommand{\Qbinv}{\text{$Q_{4,+}^{-1}$}}
\newcommand{\R}{\mathbb{R}}
\usepackage[top=3.0cm,bottom=3.0cm,left=3.0cm,right=3.0cm]{geometry}
\usepackage{lineno}
\begin{document}
%\linenumbers
\title{Some preconditioning techniques for a class of double saddle point problems}
\author{
Fariba Bakrani Balani\footnote{Department of Applied Mathematics, Faculty of Mathematical Sciences, Shahid Beheshti University, Tehran, Iran.  \texttt{E-mail}: farbakrani@gmail.com}, \
Luca Bergamaschi\footnote{Department of Civil Environmental and Architectural Engineering, University of Padua, Italy.  
\texttt{E-mail}: luca.bergamaschi@unipd.it},\ \
\'Angeles Mart\'inez\footnote{Department of Mathematics and Geosciences, University of Trieste, Trieste, Italy.  
\texttt{E-mail}: amartinez@units.it}, \ and \\
Masoud Hajarian\footnote{Department of Applied Mathematics, Faculty of Mathematical Sciences, Shahid Beheshti University, Tehran, Iran.  \texttt{E-mail}: m\_hajarian@sbu.ac.ir , 
}}
\date{}
\maketitle
\vskip -28pt
%-----------------------------------------------------
%%%%%%%%%%%%%%%%%%%%%%%%%%%%%%%%%%%%%%%%%%%%%%%
%-----------------------------------------------------
\begin{abstract}
\normalsize
In this paper, we describe and analyze the spectral properties of a number of exact block preconditioners for a class of double saddle point problems. Among all these, we consider an inexact version of a block triangular preconditioner providing extremely
	fast convergence of the FGMRES method. We develop a spectral analysis of the preconditioned matrix 
	showing
that the complex eigenvalues lie in a circle of center $(1,0)$
and radius 1, while the real eigenvalues are described in terms of the roots of a third order polynomial with real coefficients.
Numerical examples are reported to illustrate the efficiency of inexact versions of the proposed preconditioners, and to verify
the theoretical bounds.
\end{abstract}

\medskip


\smallskip


\medskip

%%%%%%%%%%%%%%%%%%%%%%%%%%%%%%%%%%%%%%%%%%%%%%
%-----------------------------------------------------
\section{Introduction\label{Asec}}
This paper is concerned with a number of block preconditioners for the numerical solution of large and sparse linear system of equations of \textit{double saddle-point} type of the form
\begin{equation}\label{Eq1}
\mathcal{A}w\equiv\begin{pmatrix}
A&B^{T}&0\\B&0&C^{T}\\0&C&0\end{pmatrix}\begin{pmatrix}x\\y\\z  \end{pmatrix}=\begin{pmatrix}f\\g
\\ h\end{pmatrix}\equiv b,
\end{equation}
where $A\in \mathbb{R}^{n \times n}$ is a symmetric positive definite matrix, $B \in \mathbb{R}^{m \times n}$ and $C \in \mathbb{R}^{l \times m}$ have full row rank, $f \in \mathbb{R}^{n}$, $g \in \mathbb{R}^{m}$ and $h \in \mathbb{R}^{l}$ are given vectors. Such linear systems arise in a number of scientific applications including constrained least squares problems \cite{Yuan}, constrained quadratic programming \cite{Han}, magma-mantle dynamics \cite{Rhebergen}, to mention a few; see, e.g. \cite{Chen, Monk, Cai}. Similar block structures
arise e.g. in liquid crystal director modeling or in the coupled Stokes-Darcy problem, and the preconditioning of such linear systems has been considered in \cite{ChenRen, Szyld, Benzi2018, BeikBenzi2022}.

Obviously, the matrix of system \eqref{Eq1} is symmetric and can be considered as a $2 \times 2$ block matrix \cite{Cao}. 
Due to the fact that these saddle point matrices are typically large and sparse, their iterative solution is recommended e.g. by Krylov subspace iterative methods~\cite{Simoncini1}. In order to improve the efficiency of iterative methods, some preconditioning techniques are employed.

To solve iteratively the linear system \eqref{Eq1}, a number of preconditioning methods have been investigated and studied in the literature. In \cite{Huang1}, Huang developed the block diagonal preconditioner $\mathcal{P}_{D}$ and its inexact version $\widehat{\mathcal{P}}_{D}$ which are of the forms
\begin{equation}\label{Eq2}
\mathcal{P}_{D}=\begin{pmatrix}
A&0&0\\0&S&0\\0&0& X \end{pmatrix}, \qquad 
\widehat{\mathcal{P}}_{D}=\begin{pmatrix}
\widehat{A}&0&0\\0&\widehat{S}&0\\0&0& \widehat X \end{pmatrix},
\end{equation}
where $S = BA^{-1}B^{T}, \ X = C S^{-1} C^T$ and  $\widehat{A}, \widehat{S}$ and $\widehat X$
are symmetric positive definite approximations of $A, S$, and $X$, respectively.  Exact and inexact versions of the block diagonal preconditioner $\widehat{\mathcal{P}}_{D}$ have been analyzed in \cite{Bradley}.
Cao \cite{Cao} considered the equivalent linear system 
\begin{equation}\label{Eq3}
\mathcal{A}w\equiv\begin{pmatrix}
A&B^{T}&0\\-B&0&-C^{T}\\0&C&0\end{pmatrix}\begin{pmatrix}x\\y\\z  \end{pmatrix}=\begin{pmatrix}f\\-g
\\ h\end{pmatrix}\equiv b,
\end{equation}
and proposed the shift-splitting iteration of the form
\begin{equation}\label{Eq4}
\frac{1}{2}(\alpha I+\mathcal{A})w^{(k+1)}=\frac{1}{2}(\alpha I-\mathcal{A})w^{(k)}+b,
\end{equation}
which leads to the preconditioner 
\begin{equation}\label{Eq5}
\mathcal{P}_{SS}=\frac{1}{2}\begin{pmatrix}
\alpha I+A&B^{T}&0\\-B&\alpha I&-C^{T}\\0&C&\alpha I\end{pmatrix},
\end{equation}
where $\alpha$ is a positive constant and $I$ is the identity matrix of appropriate size.
In addition, a relaxed version of the shift-splitting preconditioner has been considered by dropping the shift parameter in 
the (1,1) block of $\mathcal{P}_{SS}$.  In \cite{Balani-et-al} two block preconditioners  are proposed, and the spectral distributions of their inexact versions are described.

In \cite{Xie}, three exact block preconditioners for solving \eqref{Eq1} have been introduced and analyzed which are defined as 
\begin{equation}\label{Eq6}
\mathcal{P}_{1}=\begin{pmatrix}A&0&0\\B&-S&C^{T}\\0&0&-X\end{pmatrix},\quad
\mathcal{P}_{2}=\begin{pmatrix}A&0&0\\B&-S&C^{T}\\0&0& X\end{pmatrix},\quad
\mathcal{P}_{3}=\begin{pmatrix}A&B^{T}&0\\B&-S&0\\0&0&- X \end{pmatrix}.
\end{equation} 
Moreover, it is shown that the preconditioned matrices corresponding to the above preconditioners only have at most three distinct eigenvalues.  
More recently, Wang and Li \cite{Wang} have proposed an exact and inexact parameterized block symmetric positive definite preconditioner for solving the double saddle point problem \eqref{Eq1}.

Enlightened by the type $\mathcal{P}$ preconditioners  in \eqref{Eq6}, we describe several other block preconditioning approaches to be employed within a Krylov subspace methods for the solution of linear system of equations \eqref{Eq1}, 
\begin{equation}\label{Eq7}
\mathcal{Q}_{1}=\begin{pmatrix}A&B^{T}&0\\0&-S&0\\0&0& X \end{pmatrix},\quad
\mathcal{Q}_{2}=\begin{pmatrix}A&B^{T}&0\\0&S&C^{T}\\0&0&- X \end{pmatrix},\quad
\mathcal{Q}_{3}=\begin{pmatrix}A&B^{T}&0\\0&-S&C^{T}\\0&0&\pm  X \end{pmatrix},
\end{equation} 
and the block preconditioners of the forms
\begin{equation}\label{Eq8}
\mathcal{Q}_{4}=\begin{pmatrix}A&B^{T}&0\\B&0&0\\0&C&\pm  X \end{pmatrix},\quad
\mathcal{Q}_{5}=\begin{pmatrix}A&B^{T}&0\\B&0&0\\0&0&  X \end{pmatrix}.
\end{equation} 

We analyze the spectral
distribution of the corresponding preconditioned matrices which in all cases have at most three distinct eigenvalues thus
guaranteeing the finite termination of e.g. the GMRES  iterative method. In realistic problems  the proposed
preconditioners can not be used exactly since they require (for their application)
\begin{enumerate}
\item  Solution  of a system with $A$, \\[-1.8em]
\item Explicit computation of $S = B A^{-1} B^T$, \\[-1.8em]
\item Solution of a system with $S$, \\[-1.8em]
\item Explicit computation of $X = C S^{-1} C^T$, \\[-1.8em]
\item Solution of a system with $X$.
\end{enumerate}
Particularly, steps 2. and 4. require inversion of the (possibly sparse)  matrices $A$ and $S$. Practical application of the described preconditioners is then subjected to approximation of matrices $A, S$ and $X$ with $\widehat A, \widehat S$ and $\widehat X$, respectively.

To measure the effects of such approximation on the  spectral properties of the preconditioned matrix we considered the block
triangular preconditioner $\mathcal{Q}_3$ (with the plus sign in (3,3) block), and give bound on the complex and real eigenvalues
in terms of the (real and positive) eigenvalues of $\widehat A^{-1} A$,  $\widehat S^{-1} S$,  $\widehat X^{-1} X$.

The outline of this work is described as follows. In section \ref{Bsec}, we derive and analyze some exact
block preconditioners for solving double saddle point problem \eqref{Eq1}.
In section \ref{Dsec}, we test the inexact versions of the proposed preconditioners on a test case in combination
with the FGMRES iterative method. We then focus in section \ref{Csec} on the spectral analysis of the block preconditioner
which revealed the most efficient. Bound on truly complex as well as on real eigenvalues are developed for this
preconditioner and compared with the actual spectral distribution.
Some further results about the real eigenvalues of a simplified version of the triangular preconditioner are developed and 
tested in section \ref{Fsec}.
Finally, we state some conclusions in section \ref{Esec}.
%------------------------------------------------------
%%%%%%%%%%%%%%%%%%%%%%%%%%%%%%%%%%%%%%%%%%%%%%%
%-----------------------------------------------------
\section{Block preconditioners and eigenvalue analysis}\label{Bsec}
Let us consider the preconditioners in (\ref{Eq7}) and (\ref{Eq8}).
We observe that these preconditioners are nonsingular since $A$ is symmetric positive definite and both $B$ and $C$ have full row rank. We know that the eigenvalues of $\mathcal{A}\mathcal{Q}^{-1}$ and 
$\mathcal{Q}^{-1}\mathcal{A}$ (for general preconditioner $\mathcal{Q}$) are equal, so that spectral results can be given in terms of any of the two matrices.
In the following the eigenvalues of the preconditioned matrices corresponding to the proposed preconditioners are determined. The notations $\sigma(.)$ and $||.||$ denote the set of all eigenvalues of a matrix and the Euclidean norm of a vector, respectively. We use 
$\Re(\lambda)$ and $\Im(\lambda)$ to denote the real and imaginary parts of a complex eigenvalue $\lambda$.
The preconditioners $\mathcal{Q}_{3}$ and $\mathcal{Q}_{4}$ with positive (3,3)-block are denoted by $\Qa$ and $\Qb$.

\begin{Theorem}\label{Th1}
Suppose that $A\in \mathbb{R}^{n \times n}$ is symmetric positive definite and $B\in \mathbb{R}^{m \times n}$ and $C \in \mathbb{R}^{l \times m}$ are matrices with full row rank. Then the preconditioner $\mathcal{Q}_{1}$ for $\mathcal{A}$ satisfies  
$$\sigma(\mathcal{A}\mathcal{Q}_{1}^{-1}) \in \Big\{1,\frac{1}{2}(1\pm \sqrt{3}i)\Big\}.$$
\end{Theorem}
\begin{proof}
First we must compute $\mathcal{Q}_{1}^{-1}$. An easy calculations yields that 
\begin{equation}\label{Eq9}
\mathcal{Q}_{1}^{-1}=\begin{pmatrix}A^{-1}&A^{-1}B^{T}S^{-1}&0\\0&-S^{-1}&0\\0&0&X^{-1}\end{pmatrix},
\end{equation}
%where $X = CS^{-1}C^{T}$. 
It follows directly from \eqref{Eq9} that 
\begin{equation}\label{Eq10}
\mathcal{A}\mathcal{Q}_{1}^{-1}=\begin{pmatrix}I&0&0\\BA^{-1}&I&C^{T}X^{-1}\\0&-CS^{-1}&0\end{pmatrix}.
\end{equation}
We now determine the eigenvalues of the preconditioned matrix $\mathcal{A}\mathcal{Q}_{1}^{-1}$ by using the Laplace expansion. Therefore, the characteristic polynomial of $\mathcal{A}\mathcal{Q}_{1}^{-1}$ is given by
\begin{equation}\label{Eq11}
q(\lambda)=\text{det}(\lambda I- \mathcal{A}\mathcal{Q}_{1}^{-1})=(\lambda-1)^{n}
\begin{vmatrix}
(\lambda-1)I & -C^{T}X^{-1} \\ 
CS^{-1}& \lambda I
\end{vmatrix}.
\end{equation}
Clearly, $\lambda =1$ is an eigenvalue of $\mathcal{A}\mathcal{Q}_{1}^{-1}$ with algebraic multiplicity at least $n$.
To determine the rest of eigenvalues, we seek $\lambda \neq 1,\: x_{2}$ and $x_{3}$ satisfying 
\begin{align}
(\lambda-1)x_{2}-C^{T}X^{-1}x_{3}&=0, \label{Eq12-0}\\
CS^{-1}x_{2}+\lambda x_{3}&=0. \label{Eq12-1}
\end{align}
Computing $x_{2}=\frac{1}{\lambda-1}CS^{-1}x_{3}$ from equation \eqref{Eq12-0} and
substituting the value $x_{2}$ into equation \eqref{Eq12-1}, we get
\begin{equation}\label{Eq13}
\Big(\lambda^{2}I-\lambda I+I\Big)x_{3}=0
\end{equation}
Note that the vector  $x_{3}$ must be nonzero, otherwise if $x_{3} = 0$, then $x_{2}=0$, and we saw that $x_{1}=0$ if $\lambda \neq 1$. Without loss of generality, we can assume 
that $x_{3}^{*}x_{3}=1$. Multiplying equation \eqref{Eq13} on the left by $x_{3}^{*}$, we obtain 
\begin{equation}\label{Eq14}
\lambda^{2}-\lambda+1=0.
\end{equation}
The roots of \eqref{Eq14} are equal to
$\lambda=\frac{1}{2}(1\pm \sqrt{3}i)$,
which completes the proof.
\end{proof}

\begin{Remark}
From the foregoing theorem it is evident that the preconditioned matrix $\mathcal{A}\mathcal{Q}_{1}^{-1}$ has eigenvalues clustered around three values 1, $\frac{1}{2}(1+ \sqrt{3}i)$, and $\frac{1}{2}(1- \sqrt{3}i)$, therefore one can expect rapid convergence for the preconditioned GMRES method.
\end{Remark}

\begin{Remark}
It is easy to verify that the preconditioned matrix $\mathcal{T}=\mathcal{A}\mathcal{Q}_{1}^{-1}$ satisfies the following polynomial
\begin{equation*}\label{EqR1}
(\mathcal{T}-\mathcal{I})(\mathcal{T}^{2}-\mathcal{T}+\mathcal{I})=0.
\end{equation*}
Since the above relation can be factorized into distinct linear factors (over $\mathbb{R}$), we conclude that $\mathcal{T}$ is diagonalizable and has at most three distinct eigenvalues $1, \frac{1}{2}(1\pm \sqrt{3}i)$.
\end{Remark}

\begin{Theorem}\label{Th2}
Suppose that $A \in \mathbb{R}^{n \times n}$ is symmetric positive definite and $B \in \mathbb{R}^{m \times n}$ and $C \in \mathbb{R}^{l \times m}$ are matrices with full row rank. Then the preconditioner $\mathcal{Q}_{2}$ for $\mathcal{A}$ satisfies  
$$\sigma(\mathcal{A}\mathcal{Q}_{2}^{-1}) \in \Big\{\pm 1,\pm i\Big\}.$$
\end{Theorem}
\begin{proof}
Straightforward calculations reveal that
\begin{equation}\label{Eq18}
\mathcal{Q}_{2}^{-1}=\begin{pmatrix}A^{-1}&-A^{-1}B^{T}S^{-1}&-A^{-1}B^{T}S^{-1}C^{T}X^{-1}\\0&S^{-1}&S^{-1}C^{T}X^{-1}\\0&0&-X^{-1}\end{pmatrix}.
\end{equation}
It follows from \eqref{Eq18} that 
\begin{equation}\label{Eq19}
\mathcal{A}\mathcal{Q}_{2}^{-1}=\begin{pmatrix}I&0&0\\BA^{-1}&-I&-2C^{T}X^{-1}\\0&CS^{-1}&I\end{pmatrix}.
\end{equation}
To proceed we use the Laplace expansion to determine the eigenvalues of $\mathcal{A}\mathcal{Q}_{2}^{-1}$. Therefore, the characteristic polynomial of $\mathcal{A}\mathcal{Q}_{2}^{-1}$ is given by
\begin{equation}\label{Eq20}
\widehat{q}(\lambda)=\text{det}(\lambda I- \mathcal{A}\mathcal{Q}_{2}^{-1})=(\lambda-1)^{n}
\begin{vmatrix}
(\lambda+1)I & 2C^{T}X^{-1} \\ 
-CS^{-1}& (\lambda-1) I
\end{vmatrix}.
\end{equation}
It is clear that $\lambda=1$ is an eigenvalue of  $\mathcal{A}\mathcal{Q}_{2}^{-1}$ with algebraic multiplicity at least $n$.
To find the remaining eigenvalues, we seek $\lambda \neq 1,\: x_{2}$ and $x_{3}$ satisfying
\begin{align}
(\lambda+1)x_{2}+2C^{T}X^{-1}x_{3}&=0, \label{Eq21-0}\\
-CS^{-1}x_{2}+(\lambda-1) x_{3}&=0. \label{Eq21-1}
\end{align}
Notice that $x_{2} \neq 0$, otherwise if $x_{2} = 0$, then $x_{3}=0$ from equation \eqref{Eq21-1}, and we saw that $x_{1}=0$ if $\lambda \neq 1$.
From equation \eqref{Eq21-1}, we derive $x_{3}=\frac{1}{\lambda-1}CS^{-1}x_{2}$. If $x_{2}\in \text{ker}(CS^{-1})$, then $\lambda=-1$ is an eigenvalue of $\mathcal{A}\mathcal{Q}_{2}^{-1}$.  Thus, it  can be deduced that $(0;x_{2};0)$ is an eigenvector associated with $\lambda = -1$, where $x_{2}$ is an arbitrary vector. In the sequel we assume that $\lambda \neq -1$.
Substituting the value $x_{3}$ into equation \eqref{Eq21-0}, we get
\begin{equation}\label{Eq22}
\Big((\lambda^{2}-1)I+2C^{T}X^{-1}CS^{-1}\Big)x_{2}=0.
\end{equation}
We normalize $x_{2}$ such that $x_{2}^{*}x_{2}=1$, and multiply equation \eqref{Eq22} by $x_{2}$ on the left 
to obtain 
\begin{equation}\label{Eq23}
\lambda^{2}-1+2x_{2}^{*}C^{T}X^{-1}CS^{-1}x_{2}=0.
\end{equation}
From the above relation, the eigenvalue $\lambda$ can be expressed 
\begin{equation}\label{Eq24}
\lambda=\pm \sqrt{1-2x_{2}^{*}C^{T}X^{-1}CS^{-1}x_{2}}.
\end{equation}
It is easy to verify that $C^{T}X^{-1}CS^{-1}$ is a projector onto $\mathcal{R}(C^{T}X^{-1})=\mathcal{R}(C^{T})$, where $\mathcal{R}$ denotes the range of a matrix. We rewrite the relation \eqref{Eq21-0} as
\begin{equation}\label{Eq16}
x_{2}=\Big(\frac{-2}{\lambda+1}\Big)C^{T}X^{-1}x_{3},
\end{equation} 
and hence we observe $x_{2} \in \mathcal{R}(C^{T}X^{-1})$. Consequently, we have 
\begin{equation}\label{Eq17}
x_{2}^{*}C^{T}X^{-1}CS^{-1}x_{2}=1. 
\end{equation} 
It follows that $\mathcal{A}\mathcal{Q}_{2}^{-1}$ has eigenvalues $\lambda =\pm i$, and we have proved the theorem.
\end{proof}

\begin{Remark}\label{Re3}
From equations \eqref{Eq21-0}, \eqref{Eq21-1} and  the proof of Theorem \ref{Th2}, it can be seen that $\lambda=-1$ may or may not be an eigenvalue of $\mathcal{A}\mathcal{Q}_{2}^{-1}$. We observed that $\lambda=-1$ is an eigenvalue if $x_2 \in \text{ker}(CS^{-1})$. Conversely, suppose that $\lambda=-1$ is an eigenvalue of $\mathcal{A}\mathcal{Q}_{2}^{-1}$. From Eqs. \eqref{Eq21-0} and \eqref{Eq21-1}, we have $x_{3}=0$ and then $CS^{-1}x_{2}=0$ which means that  $x_2 \in \text{ker}(CS^{-1})$.
This condition is necessary and sufficient for $\lambda=-1$ to be an eigenvalue of $\mathcal{A}\mathcal{Q}_{2}^{-1}$ with associated eigenvector of the form $(0;x_{2};0).$
\end{Remark}

\begin{Remark}\label{Re4}
It is easy to check that the preconditioned matrix $\mathcal{F}=\mathcal{A}\mathcal{Q}_{2}^{-1}$
satisfies
\begin{equation}\label{EqR2}
(\mathcal{F}-\mathcal{I})(\mathcal{F}+\mathcal{I})(\mathcal{F}^{2}+\mathcal{I})=0.
\end{equation}
From the relation \eqref{EqR2}, we can conclude that $\mathcal{F}$ is diagonalizable and has at most four distinct eigenvalues $\pm1, \pm i$.
\end{Remark}

\begin{Remark}
According to the following partitioning  
\begin{equation*}
\left(
\begin{array}{cc|c}
*&*&0\\ *&*&*\\ \hline
0&0&*
\end{array}
\right),
\end{equation*}
it is clear that the preconditioners $\mathcal{P}_{1}$, $\mathcal{P}_{2}$, $\mathcal{Q}_{3}$ and 
$\Qa$ have the same structure and all are in the class of block triangular preconditioners. The eigenvalues of the preconditioned matrices $\mathcal{A}\mathcal{Q}_{3}^{-1}$ and $\mathcal{A}\Qainv  $ are also provided here for completeness.
\end{Remark}

\begin{Theorem}\label{Th3}
Suppose that $A \in \mathbb{R}^{n \times n}$ is symmetric positive definite and $B \in \mathbb{R}^{m \times n}$ and $C \in \mathbb{R}^{l \times m}$ have full row rank. Then the eigenvalues of the preconditioned matrices  $\mathcal{A}\mathcal{Q}_{3}^{-1}$ and $\mathcal{A}\mathcal{Q}_{4}^{-1}$ are 1 and -1.
\end{Theorem}
\begin{proof}
After straightforward computations, we can obtain
\begin{equation*}\label{Eq25}
\begin{aligned}
\mathcal{A}\mathcal{Q}_{3}^{-1}&=
\begin{pmatrix}A&B^{T}&0\\B&0&C^{T}\\0&C&0\end{pmatrix}
\begin{pmatrix}A^{-1}&A^{-1}B^{T}S^{-1}&A^{-1}B^{T}S^{-1}C^{T}X^{-1}\\0&-S^{-1}&-S^{-1}C^{T}X^{-1}\\0&0&-X^{-1}\end{pmatrix}\\
&=\begin{pmatrix}I&0&0\\BA^{-1}&I&0\\0&-CS^{-1}&-I\end{pmatrix},
\end{aligned}
\end{equation*}
and
\begin{equation*}\label{Eq26}
\begin{aligned}
\mathcal{Q}_{4}^{-1}\mathcal{A}&=
\begin{pmatrix}A^{-1}-A^{-1}B^{T}S^{-1}BA^{-1}&A^{-1}B^{T}S^{-1}&0\\S^{-1}BA^{-1}&-S^{-1}&0
\\X^{-1}CS^{-1}BA^{-1}&-X^{-1}CS^{-1}&-X^{-1}\end{pmatrix}
\begin{pmatrix}A&B^{T}&0\\B&0&C^{T}\\0&C&0\end{pmatrix}\\
&=\begin{pmatrix}I&0&A^{-1}B^{T}S^{-1}C^{T}\\0&I&-S^{-1}C^{T}\\0&0&-I\end{pmatrix},
\end{aligned}
\end{equation*}
which implies that preconditioned matrices  $\mathcal{A}\mathcal{Q}_{3}^{-1}$ and $\mathcal{Q}_{4}^{-1}\mathcal{A}$ (or $\mathcal{A}\mathcal{Q}_{4}^{-1}$) have eigenvalues 1 and -1.
\end{proof}

\begin{Remark}\label{Re5}
It is easy to check that the minimum polynomial of the preconditioned matrices $\mathcal{A}\mathcal{Q}_{3}^{-1}$ and $\mathcal{Q}_{4}^{-1}\mathcal{A}$ have order 4.  
One may imply favorable convergence rate for the Krylov subspace methods.
\end{Remark}

\begin{Remark} It is obvious that the preconditioned matrices $\mathcal{A}\Qainv$ and $\Qbinv\mathcal{A}$ satisfy
\begin{equation*}\label{Eq26}
\mathcal{A}\Qainv=\begin{pmatrix}I&0&0\\BA^{-1}&I&0\\0&-CS^{-1}&I\end{pmatrix}, \quad  \text{and} \quad
\Qbinv\mathcal{A}=\begin{pmatrix}I&0&A^{-1}B^{T}S^{-1}C^{T}\\0&I&-S^{-1}C^{T}\\0&0&I\end{pmatrix}.
\end{equation*}
Hence, we conclude that the eigenvalues of $\mathcal{A}\Qainv$ and 
$\Qbinv\mathcal{A}$ are all one. Moreover, the minimum polynomial of the preconditioned matrix $\mathcal{A}\Qainv$ has order 3, while 
$\Qbinv\mathcal{A}$ has minimum polynomial of order  2. Therefore, the GMRES method will reach the exact solution in at most three and two steps, respectively.
\end{Remark}

\begin{Theorem}\label{Th5}
Suppose that $A\in \mathbb{R}^{n \times n}$ is symmetric positive definite and $B\in \mathbb{R}^{m \times n}$ and $C \in \mathbb{R}^{l \times m}$ are matrices with full row rank. Then the preconditioner $\mathcal{Q}_{5}$ for $\mathcal{A}$ satisfies  
$$\sigma(\mathcal{A}\mathcal{Q}_{5}^{-1}) \in \Big\{1,\frac{1}{2}(1\pm \sqrt{3}i)\Big\}.$$
\end{Theorem}
\begin{proof}
By simple calculations, we can obtain 
\begin{equation*}\label{Eq27}
\mathcal{Q}_{5}^{-1}=\begin{pmatrix}A^{-1}-A^{-1}B^{T}S^{-1}BA^{-1}&A^{-1}B^{T}S^{-1}&0\\S^{-1}BA^{-1}&-S^{-1}&0\\0&0&X^{-1}\end{pmatrix}.
\end{equation*}
We readily verify that  
\begin{equation*}\label{Eq28}
\mathcal{A}\mathcal{Q}_{5}^{-1}=\begin{pmatrix}I&0&0\\0&I&C^{T}X^{-1}\\CS^{-1}BA^{-1}&-CS^{-1}&0\end{pmatrix}.
\end{equation*}
The rest of the proof is similar to that of Theorem \ref{Th1}; we omit the details here. 
It can also be shown that the preconditioned matrix $\mathcal{J}=\mathcal{A}\mathcal{Q}_{5}^{-1}$ satisfies
\begin{equation*}\label{EqR3}
(\mathcal{J}-\mathcal{I})(\mathcal{J}^{2}-\mathcal{J}+\mathcal{I})=0,
\end{equation*}
and it follows that  $\mathcal{J}$ is diagonalizable and has at most three distinct eigenvalues $1, \frac{1}{2}(1\pm \sqrt{3}i)$.
\end{proof}

Finally, we mention that although the preconditioners type $\mathcal{P}$ and $\mathcal{Q}$ are theoretically powerful, they are not practical in real-world problems. Solving linear systems involving $S$ and $X$  can be prohibitive.
However, we may be able to devise effective inexact version of these preconditioners for $\mathcal{A}$ by approximating
the Schur complement matrices.
To apply the proposed type $\mathcal{Q}$ preconditioners within a Krylov subspace method,
we require to solve the linear system of equations with the coefficient matrices $A$, 
$S=BA^{-1}B^{T}$ and $X=CS^{-1}C^{T}$. 
Instead of exactly performing solves with $S$ and $X$, we approximate $S$ and $X$ by matrices 
$\widehat S$ and $\widehat X$, respectively, and then the linear systems involving these matrices can be iteratively solved.

\section{Experimental comparisons of the preconditioners}\label{Dsec}
In this section we present a numerical test, which will be solved by 
Flexible GMRES (FGMRES) with no restart, using the previously described preconditioners. The comparisons will be carried on in terms of 
number of FGMRES iterations (represented by ITS) and elapsed CPU time in seconds (represented by CPU).
We also provide the norm of error vectors denoted as $\text{ERR}=\|w^{(k)}-w^{*}\|_{2}/\|w^{*}\|_{2}$. 
The initial guess is set to be the zero vector and the iterations will be stopped whenever 
$$\|b-\mathcal{A}w^{(k)}\|_{2}/\|b\|_{2} < \texttt{tol},$$
with the tolerance $\texttt{tol}$ selected as will be explained below.

In all cases, the right-hand side vector $b$ is computed after selecting the exact solution of \eqref{Eq1} as
\begin{enumerate}
\item 		$w= e \equiv (1,1,\ldots,1)^{T} \in \mathbb{R}^{n+m+l}$.
\item 		$w$ a random vector of the appropriate dimension.
\end{enumerate}

The numerical experiments presented in this work have been carried out on a computer with an 
Intel Core i7-1185G7 CPU @ 3.00GHz processor and 16 GB RAM using $\textsc{Matlab}$ 2022a.

\begin{Example}(\cite{Huang1,Xie})\label{Ex2}
We consider the linear system of equations \eqref{Eq1} for which
\begin{equation*}
A=\text{diag}\begin{pmatrix}2W^{T}W+D_{1},D_{2},D_{3}\end{pmatrix}\in \mathbb{R}^{n \times n},
\end{equation*}
is a block diagonal matrix;
\begin{equation*}
B=[E,-I_{2p_{1}},I_{2p_{2}}]\in \mathbb{R}^{m \times n},\quad \text{and}\quad C = E^{T}\in \mathbb{R}^{l \times m},
\end{equation*}
are both full row rank matrices where $p_{1} = p^{2},\: p_{2} = p(p + 1)$; $W = (w_{ij}) \in \mathbb{R}^{p_{2} \times p_{2}}$ with 
$w_{ij} = e^{-2((i/3)^{2}+(j/3)^{2})}$; $D_{1}= I_{p_{2}}$ is the identity matrix;
$D_{i}= \text{diag}(d^{(i)}_{j}) \in \mathbb{R}^{2p_{1} \times 2p_{1}},\:( i=2,3)$
are diagonal matrices with
\begin{equation*}
\begin{aligned}
&d^{(2)}_{j}=\begin{cases}
1, \qquad &\text{for} \quad 1\leq j \leq p_{1}, \\
10^{-5}(j-p_{1}^{2}), \qquad &\text{for} \quad p_{1}+1\leq j \leq 2p_{1},
\end{cases}\\
&d^{(3)}_{j}=10^{-5}(j+p_{1}^{2}), \qquad \text{for}\quad  1\leq j \leq 2p_{1};
\end{aligned}
\end{equation*}
and 
\begin{equation*}
E=\begin{pmatrix} E_{1}\otimes I_{p}\\  I_{p}\otimes E_{1} \end{pmatrix}, \qquad \text{with} \qquad
E_{1}=\begin{pmatrix} 2&-1&\\ &2&-1\\& & \ddots & \ddots\\ & & & 2&-1 \end{pmatrix}\in \mathbb{R}^{p \times (p+1)}.
\end{equation*}
To summarize we have $n = 5p^2 + p$, $m = 2p^2$ and $l = p^2 + p$ and the size of the double saddle point matrix is $n+m+l = 8p^2 + 2p$.
\end{Example}

As the block approximations, we take $\widehat S=\text{tridiag}(B\widehat A^{-1}B^{T})$, the tridiagonal part of $S$,
where $\widehat A=\text{diag}(A)$ and
$\widehat X=C L_S^{-T} L_S^{-1} C^{T}$ with $L_S$ the exact bidiagonal factor of $\widehat S$.

In addition to the previously analyzed preconditioners we also present the results of an inexact variant of 
the block preconditioner considered in \cite{AslSalBei} (here denoted as $\mathcal{P}_{ASB}$), 
where the Schur complement  matrix is simply approximate with the identity matrix ($\widehat S \equiv I$),
showing outstanding results in the solution of this Example.  
However, these results (convergence in 2 iterations) can be obtained only by the choice of the right-hand-side ($b = A e$) 
and of the initial (zero) vector. In this case the initial residual $r_0 \equiv b$ satisfies
$\mathcal A r_0 = \mathcal {P}_{ASB} r_0$ which makes the Krylov subspace of degree 1 invariant for this particular 
right-hand-side  and therefore ensures convergence in at most two iterations. Clearly, changing the right-hand-side, 
this property does no longer hold.
\begin{figure}[H]\centering
\begin{minipage}{6.4cm}
\includegraphics[width=6.4cm]{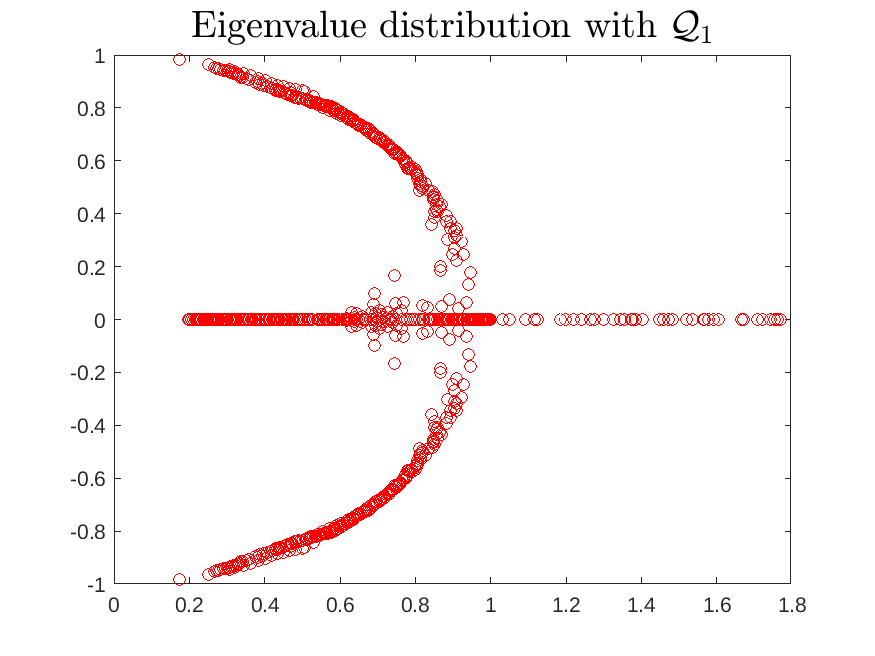}
\end{minipage}
\hspace{-7mm}
\begin{minipage}{6.4cm}
\includegraphics[width=6.4cm]{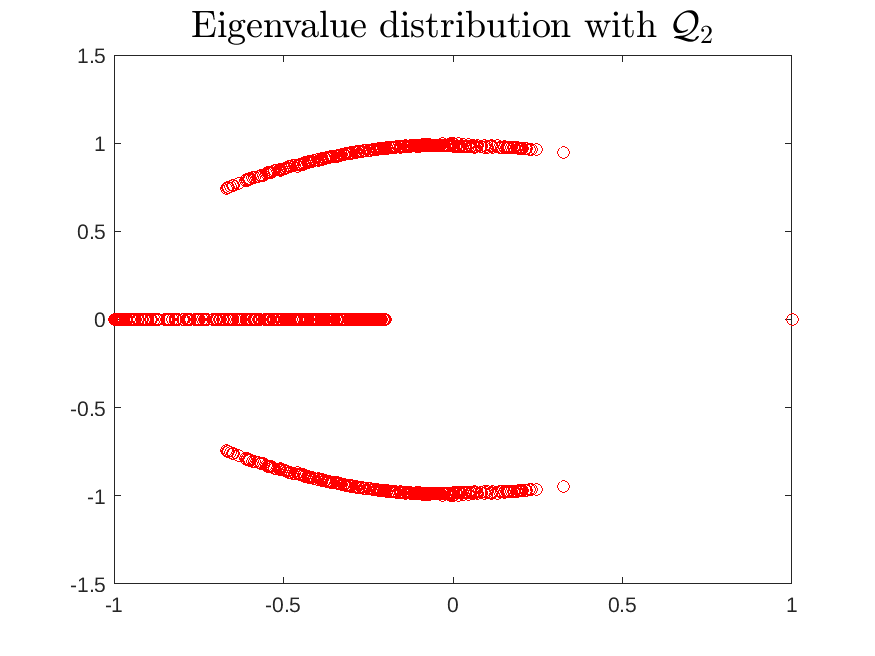}
\end{minipage}

\begin{minipage}{6.4cm}
\includegraphics[width=6.4cm]{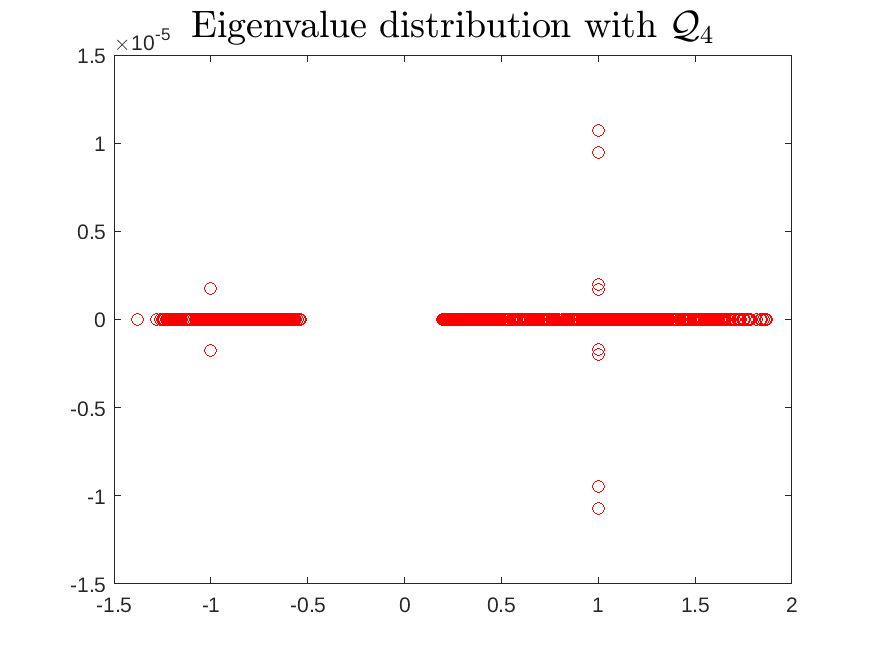}
\end{minipage}
\hspace{-7mm}
\begin{minipage}{6.4cm}
\includegraphics[width=6.4cm]{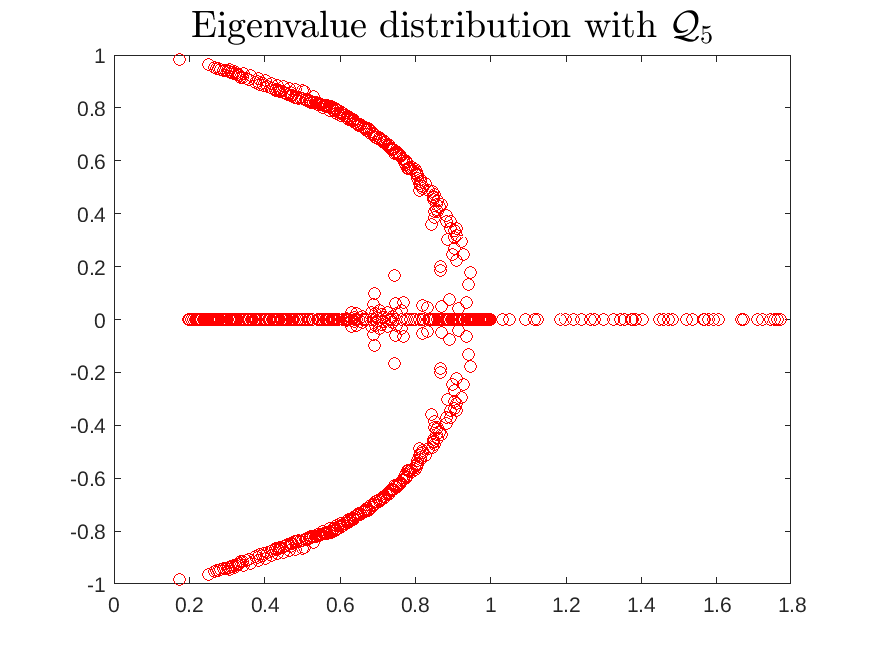}
\end{minipage}

\begin{minipage}{6.4cm}
\includegraphics[width=6.4cm]{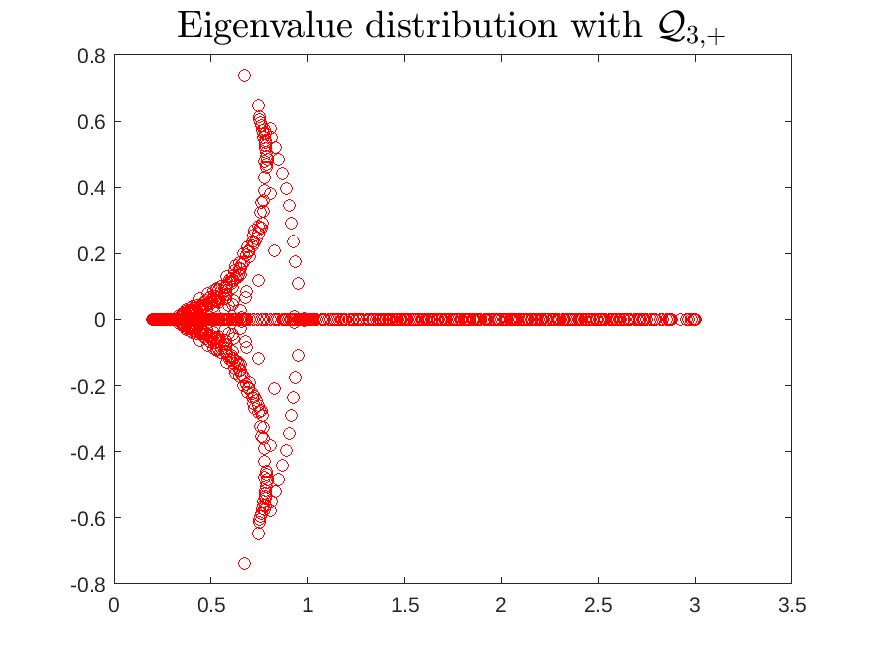}
\end{minipage}
\hspace{-7mm}
\begin{minipage}{6.4cm}
\includegraphics[width=6.4cm]{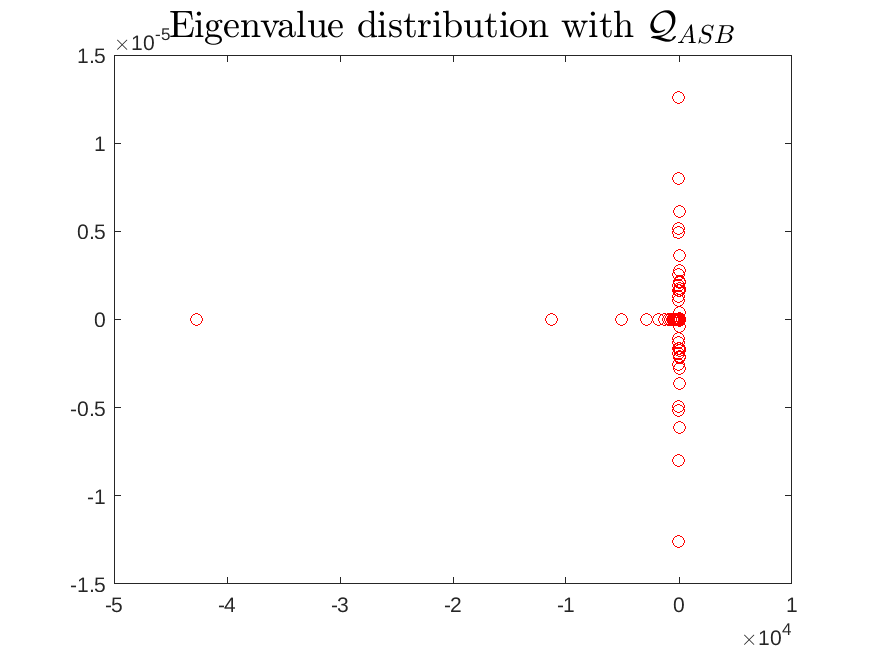}
\end{minipage}
%%%%{\includegraphics[scale=0.56]{Inexact_eig2}}\vspace{-0.5cm}
\caption{Eigenvalue distribution of the preconditioned matrix with
	$p=16$ for  different preconditioners.}
\label{fig2-2}
\end{figure}

Figure \ref{fig2-2} plots the eigenvalue distribution of the preconditioned matrices with the approximation matrices $\widehat S$ and $\widehat X$.  We observe from this figure that the 
preconditioned matrix with $\Qa$ has more clustered eigenvalues than the other ones, which can considerably improve the convergence rate of the Krylov subspace iterative methods.
Preconditioners ${\mathcal Q}_1$  and ${\mathcal Q}_5$  display quite similar eigenvalue distributions 
(the ideal version have the same eigenvalues). We then remove ${\mathcal Q}_1$ from the numerical results
in view of its slightly higher application cost in comparison with ${\mathcal Q}_5$.
All other proposed preconditioners display also negative eigenvalues, predicting  slow GMRES convergence.
Finally, the spectral distribution with the ${\mathcal P}_{ASB}$ preconditioner, does not seem to be favorable, as it spreads
over a wide (real) interval.

The linear system with $\widehat S$ is solved exactly by solving two bidiagonal systems with $L_S$ and $L_S^T$,
while the system with $\widehat X$ is solved, without forming explicitly matrix $\widehat X$,  by the PCG method accelerated 
by the incomplete Cholesky factorization of $C \text{diag}(\widehat S)^{-1} C^T$ with a 
drop tolerance $\tau = 10^{-4}$.  The work to be done before the beginning of the FGMRES process is described in 
Algorithm \ref{Algo0} while the application of the preconditioner at each FGMRES iteration
is sketched in Algorithm \ref{Algo01}, for the 
preconditioner $\Qa$.

\begin{algorithm}[h!]
\begin{algorithmic}[1]
	\caption{Approximations of the blocks before the FGMRES solution. }
	\label{Algo0}
	\STATE $\widehat A = \text{diag}(A)$, 
	\STATE $\widehat{S} = \text{tridiag}(B \widehat{A}^{-1}B^T) $
	\STATE $L_S$ is the exact bidiagonal factor of $\widehat{S}$ such that $\widehat{S} = L_S L_S^T$.
	\STATE Compute $X_0 = C \text{diag}(\widehat S)^{-1} C^T$ and $M$ its IC factor with drop tolerance $10^{-4}$.
\end{algorithmic}
\end{algorithm}

\begin{algorithm}[H]
\begin{algorithmic}[1]
	\caption{Computation of $w = \Qainv r$. }
	\label{Algo01}
	\STATE Split vector $r$ into $r = \begin{bmatrix}r_1^T & r_2^T & r_3^T \end{bmatrix}^T$ \\[.1em]
%			\STATE Compute the incomplete Cholesky factor $M$ of $C \text{diag}(\widehat S)^{-1} C^T$. \\[.1em]
	\STATE 
	Solve system $C L_S^{-T} L_S^{-1} C^T w_3 = r_3$ by PCG with preconditioner $M M^T$ and 
	$\texttt{tol}_{\texttt{PCG}} = 10^{-4}$\\[.1em]
	\STATE $v = C^T w_3 - r_2$ \\[.1em]
	\STATE $w_2 = L_S^{-T} L_S^{-1} v$ \\[.1em]
	\STATE $y = r_1 - B^T w_2$ \\[.1em]
	\STATE Solve $A w_1 = y$ \\[.1em]
	\STATE Define  $w = \begin{bmatrix}w_1^T & w_2^T & w_3^T \end{bmatrix}^T$
\end{algorithmic}
\end{algorithm}

\begin{table}[h!]
%	\begin{small}
        \caption{Numerical results for example \ref{Ex2} with unitary exact solution.}
	\label{ones}
\begin{tabular}{lc|rrrrrrrr}
        &size     & 2080 & 8256   &  32896   &  131328  & 524800 & 2\,098176 & 8\,390656 \\
        &$p$      & 16 & 32   &  64   &  128 & 256 & 512 & 1024  \\
          \hline
	$\mathcal {P}_D$ 	&ITS&   79 &  126 &  132 &  128 &  121 &  117 & \dag \\
	&CPU&    0.15 &    0.83 &    3.28 &   13.25 &   49.90 &  165.06 & \dag \\
	&RES&  0.21e$-05$&  0.12e$-06$&  0.88e$-08$&  0.49e$-09$&  0.35e$-10$&  0.20e$-11 $ & \dag \\
	&ERR&  0.88e$-05$&  0.64e$-05$&  0.10e$-04$&  0.11e$-04$&  0.12e$-04$&  0.11e$-04 $ & \dag \\
	\hline
	$\mathcal {P}_3$ &ITS&   51 &   83 &   87 &   84 &   80 &   77 & \dag \\
	&CPU&    0.07 &    0.43 &    1.62 &    7.03 &   26.45 &   88.69 & \dag \\
	&RES&  0.19e$-05$&  0.12e$-06$&  0.81e$-08$&  0.53e$-09$&  0.34e$-10$&  0.19e$-11 $ & \dag \\
	&ERR&  0.68e$-05$&  0.60e$-05$&  0.10e$-04$&  0.11e$-04$&  0.12e$-04$&  0.11e$-04 $ & \dag \\
%%	\hline
%%	$\mathcal {Q}_1$ &ITS&   40 &   59 &   61 &   59 &   56 &   54 & 52 \\
%%	&CPU&    0.04 &    0.24 &    0.93 &    4.10  &   15.36 &   51.46 & 317.15\\
%%	&RES&  0.22e$-05$&  0.12e$-06$&  0.90e$-08$&  0.56e$-09$&  0.34e$-10$&  0.20e$-11 $ & 0.12e$-12$\\
%%	&ERR&  0.72e$-05$&  0.59e$-05$&  0.13e$-04$&  0.14e$-04$&  0.14e$-04$&  0.13e$-04 $ & 0.11e$-04$\\
	\hline
	$\mathcal {Q}_2$ &ITS&   66 &   98 &  100 &   99 &   95 &   91 & \dag \\
	&CPU&    0.10  &    0.64 &    1.94 &    8.33 &   31.30 &  103.49 & \dag \\
	&RES&  0.21e$-05$&  0.13e$-06$&  0.92e$-08$&  0.57e$-09$&  0.34e$-10$&  0.22e$-11 $ & \dag \\
	&ERR&  0.66e$-05$&  0.64e$-05$&  0.97e$-05$&  0.10e$-04$&  0.10e$-04$&  0.10e$-04 $ & \dag \\
	\hline
	$\mathcal {Q}_4$ &ITS&   48 &   81 &   85 &   82 &   79 &   76 & \dag \\
	&CPU&    0.06 &    0.46 &    1.56 &    6.72 &   25.59 &   85.84 & \dag \\
	&RES&  0.22e$-05$&  0.14e$-06$&  0.88e$-08$&  0.57e$-09$&  0.35e$-10$&  0.20e$-11 $ & \dag \\
	&ERR&  0.64e$-05$&  0.67e$-05$&  0.10e$-04$&  0.12e$-04$&  0.11e$-04$&  0.10e$-04 $ & \dag \\
	\hline
	$\mathcal{Q}_5$ &ITS&   38 &   57 &   59 &   57 &   54 &   52 & 52 \\
	&CPU&    0.04 &    0.22 &    0.93 &    3.90 &   14.49 &   49.81 & 294.22 \\
	&RES&  0.21e$-05$&  0.14e$-06$&  0.86e$-08$&  0.53e$-09$&  0.34e$-10$&  0.19e$-11 $ & 0.12e$-12$\\
	&ERR&  0.59e$-05$&  0.68e$-05$&  0.10e$-04$&  0.11e$-04$&  0.12e$-04$&  0.10e$-04 $ & 0.11e$-04$\\
	\hline
	$\Qa$ &ITS&   30 &   44 &   46 &   45 &   43 &   41 & 39 \\
	&CPU&    0.03 &    0.17 &    0.64 &    2.68 &   10.40 &   35.52 & 201.87\\
	&RES&  0.19e$-05$&  0.14e$-06$&  0.86e$-08$&  0.47e$-09$&  0.33e$-10$&  0.20e$-11 $ & 0.12e$-12$ \\
	&ERR&  0.88e$-05$&  0.69e$-05$&  0.13e$-04$&  0.12e$-04$&  0.14e$-04$&  0.13e$-04 $ & 0.15e$-04$\\
	\hline
	$\mathcal {P}_{ASB}$ &ITS&   19 &   17 &   16 &   15 &   13 &   12 & 45 \\
	&CPU&    0.05 &    0.03 &    0.07 &    0.27 &    0.84 &    2.85 &  177.30 \\
	&RES&  0.15e$-05$&  0.12e$-06$&  0.82e$-08$&  0.38e$-09$&  0.30e$-10$&  0.16e$-11 $ & 0.13e$-12$ \\
	&ERR&  0.39e$-05$&  0.31e$-05$&  0.33e$-05$&  0.25e$-05$&  0.32e$-05$&  0.26e$-05 $ & 0.56e$-05$\\
	\hline
 \hline
         \multicolumn{9}{l}{\dag: GMRES memory overflow}
\end{tabular}
%	\end{small}
\end{table}
%\end{document}

The numerical results corresponding to the block preconditioned FGMRES for Example \ref{Ex2} are given in Tables 
\ref{ones} (using the right hand side as $b = \mathcal A e$)  and \ref{random} where an exact random solution is employed.

We run FGMRES with all preconditioners for problems with $p = \{16, 32, 64, 128, 256, 512, 1024\}$ ending up with 
a problem with more than $8$ million unknowns. To obtain a relative error of (roughly) the same order of magnitude we adjusted
the tolerance as 
\[\texttt{tol} = \dfrac{10}{N^2}, \qquad N = n+m+l.\]

\begin{table}[h!]
	\caption{Numerical results for Example \ref{Ex2} with random exact solution.}
	\label{random}
\begin{tabular}{lc|rrrrrrrr}
	&size	  & 2080 & 8256   &  32896   &  131328  & 524800 & 2\,098176 & 8\,390656 \\
	&$p$	  & 16 & 32   &  64   &  128 & 256 & 512 & 1024  \\
	  \hline
	$\mathcal{P}_{D}$  &ITS &     89 &  147 &  156 &  156 &  152 &  150 & \dag \\
	&CPU & 0.13 &    0.90 &    4.31 &   16.58 &   61.67 &  257.74& \dag  \\
	&RES & 0.20e$-05$ &  0.14e$-06$ &  0.85e$-08$ &  0.51e$-09$ &  0.35e$-10$ &  0.21e$-11$& \dag  \\
	&ERR &   0.83e$-05$ &  0.71e$-05$ &  0.73e$-05$ &  0.73e$-05$ &  0.87e$-05$ &  0.83e$-05$ & \dag \\
   \hline
	$\mathcal{P}_{3}$  &ITS & 58 &   97 &  103 &  101 &  100 &   97 & \dag \\
	&CPU & 0.06 &    0.48 &    1.98 &    8.32 &   32.34 &  128.20 &\dag  \\
	&RES &0.20e$-05$ &  0.11e$-06$ &  0.83e$-08$ &  0.57e$-09$ &  0.33e$-10$ &  0.22e$-11$& \dag  \\
	&ERR &0.87e$-05$ &  0.51e$-05$ &  0.72e$-05$ &  0.80e$-05$ &  0.78e$-05$ &  0.84e$-05$& \dag  \\
%%   \hline
%%	$\mathcal{Q}_1$ &ITS & 45 &   68 &   71 &   70 &   68 &   67  & \dag  \\
%%	&CPU &   0.04 &    0.26 &    1.05 &    4.88 &   19.19 & 73.30  &  \dag \\
%%	&RES &0.19e$-05$ &  0.12e$-06$ &  0.80e$-08$ &  0.49e$-09$ &  0.34e$-10$ & 0.22e$-11$ & \dag  \\
%%	&ERR &0.58e$-05$ &  0.54e$-05$ &  0.69e$-05$ &  0.71e$-05$ &  0.85e$-05$ & 0.88e$-05$ & \dag  \\
   \hline
	$\mathcal{Q}_2$&ITS & 74 &  114 &  118 &  117 &  114 &  112 & \dag \\
	&CPU & 0.07 &    0.61 &    2.21 &    9.67 &   38.40 &  146.92 & \dag \\
	&RES & 0.21e$-05$ &  0.13e$-06$ &  0.87e$-08$ &  0.52e$-09$ &  0.35e$-10$ &  0.22e$-11$& \dag  \\
	&ERR &0.59e$-05$ &  0.57e$-05$ &  0.71e$-05$ &  0.70e$-05$ &  0.81e$-05$ &  0.83e$-05$ & \dag \\

   \hline
	$\mathcal{Q}_4$&ITS & 55 &   96 &  103 &  103 &  100 &   98& \dag  \\
	&CPU & 0.06 &    0.50 &    1.93 &    8.46 &   33.27 &  126.98& \dag  \\
	&RES & 0.22e$-05$ &  0.14e$-06$ &  0.79e$-08$ &  0.47e$-09$ &  0.34e$-10$ &  0.22e$-11$& \dag  \\
	&ERR &0.67e$-05$ &  0.63e$-05$ &  0.66e$-05$ &  0.65e$-05$ &  0.79e$-05$ &  0.80e$-05$& \dag  \\
   \hline
	$\mathcal{Q}_5$&ITS & 43 &   66 &   69 &   68 &   66 &   65  & 60\\
	&CPU & 0.04 &    0.29 &    1.01 &    4.70 &   18.38 &   70.22 & 450.92  \\
	&RES &    0.18e$-05$ &  0.12e$-06$ &  0.80e$-08$ &  0.50e$-09$ &  0.35e$-10$ &  0.22e$-11$ & 0.14e$-12$ \\
	&ERR &0.57e$-05$ &  0.54e$-05$ &  0.67e$-05$ &  0.69e$-05$ &  0.83e$-05$ &  0.86e$-05$ & 0.85e$-05$\\
   \hline
	$\Qa$&ITS & 33 &   51 &   54 &   53 &   52 &   52  & 51 \\
	&CPU & 0.03 &    0.21 &    0.69 &    3.14 &   13.17 &   49.47  & 301.68\\
	&RES & 0.20e$-05$ &  0.13e$-06$ &  0.75e$-08$ &  0.49e$-09$ &  0.33e$-10$ &  0.18e$-11$ & 0.18e$-12$ \\
	&ERR &0.11e$-04$ &  0.57e$-05$ &  0.65e$-05$ &  0.73e$-05$ &  0.81e$-05$ &  0.71e$-05$ & 0.78e$-05$\\
   \hline
	$\mathcal{P}_{ASB}$&ITS & 66 &   98 &  113 &  117 &  117 &  116& \dag  \\
	&CPU & 0.06 &    0.41 &    1.67 &    7.54 &   29.38 &  119.92& \dag  \\
	&RES & 0.22e$-05$ &  0.14e$-06$ &  0.77e$-08$ &  0.56e$-09$ &  0.32e$-10$ &  0.21e$-11$& \dag  \\
	&ERR &0.65e$-05$ &  0.96e$-05$ &  0.16e$-04$ &  0.22e$-04$ &  0.21e$-04$ &  0.22e$-04$ & \dag \\
   \hline
	 \multicolumn{9}{l}{\dag: GMRES memory overflow}
\end{tabular}
\end{table}

From these tables, we see that the preconditioners $\mathcal{Q}_{1}$, 
$\mathcal{Q}_5$ and, particularly,  $\Qa$ outperform the 
other ones in terms of iteration number and CPU time, being all the proposed preconditioners more convenient
than $\mathcal{P}_{D}$
and obtaining FGMRES convergence to the solution of \eqref{Eq1} in  a reasonable number
of iterations and CPU time. For the largest problem and random right-hand-side, only preconditioners
$\Qa$ and $\mathcal{Q}_5$ could solve the given linear system within the memory of our laptop, due do the small
number of iterations they required. 

Regarding the $\mathcal{P}_{ASB}$ preconditioner we observe that it is the most performing one when $w \equiv e$, yet
not providing convergence in two iterations since its exact version is employed, whereas it reveals not competitive
with $\Qa$ for a random exact solution.

In the next section, we will discuss in detail the eigenvalue distribution of the preconditioner matrix using the inexact version of $\Qa$ and we leave other inexact version of the
proposed preconditioners as a topic for further research. Regarding the complex
eigenvalues, we perform an analysis similar to the one in~\cite{Simoncini}, but generalized here for the $3\times 3$ block triangular preconditioner.
%-----------------------------------------------------

%%%%%%%%%%%%%%%%%%%%%%%%%%%%%%%%%%%%%%%%%%%%%%%%
%----------------------------------------------------
\section{Eigenvalue analysis of the  inexact variants of $\Qa$ \label{Csec}}
We analyze in this section the eigenvalue distribution  of the preconditioned matrix 
$\mathcal{A}\bar{\mathcal{Q}}^{-1}$, where, in the sequel,
\begin{equation}\label{Eq29}
\bar{\mathcal{Q}}\equiv \Qa =\begin{pmatrix} \widehat A &B^{T}&0\\0&-\widehat S &C^{T}\\0&0& \widehat X\end{pmatrix},
\end{equation}
with $ \widehat A, \widehat S$ and $\widehat X$ proper SPD approximations (preconditioners) of $A$, $S$ and $X$, respectively. 
%Note that $\bar {\mathcal {Q}}$ is  the inexact counterpart of ${\mathcal Q}_3$, defined in (\ref{Eq7}), with positive $(3,3)$ block.

The relevant  spectral properties of the preconditioned matrix $\mathcal{A}\bar{\mathcal{Q}}^{-1}$ will be given in terms of the eigenvalues of 
$\widehat A^{-1} A, \widehat S^{-1} \tilde{S}$ and $\widehat X^{-1} \tilde{X}$ where $\tilde{S} =  B \widehat A^{-1} B^T$ and $\tilde{X} = C \widehat S^{-1} C^T$. To this aim, we define
\begin{eqnarray}  
\gamma_{\min}^A  \equiv \lambda_{\min} (\widehat A^{-1} A), &\qquad & \gamma_{\max}^A  \equiv \lambda_{\max} (\widehat A^{-1} A), \qquad \gamma_A \in [ \gamma_{\min}^A,  \gamma_{\max}^A ],  \\
\gamma_{\min}^{S}  \equiv \lambda_{\min} (\widehat S^{-1} \tilde{S}), &\qquad & \gamma_{\max}^{S}  \equiv \lambda_{\max} (\widehat S^{-1} \tilde{S}), \qquad \gamma_{S} \in [ \gamma_{\min}^{S},  \gamma_{\max}^{S} ],  \\
\gamma_{\min}^{X}  \equiv \lambda_{\min} (\widehat X^{-1} \tilde{X}), &\qquad & \gamma_{\max}^{X} \equiv \lambda_{\max} (\widehat X^{-1} \tilde{X}), \qquad \gamma_{X} \in [ \gamma_{\min}^{X},  \gamma_{\max}^{X} ].  
\end{eqnarray}
We will finally make the assumption that $ 1\in [\gamma_{\min}^A, \gamma_{\max}^A$]. This assumption, very commonly satisfied in
practice, will simplify some of the bounds mostly regarding  real eigenvalues.

\noindent
Let 
\begin{equation}\label{Eq30}
\bar{\mathcal{D}}=\begin{pmatrix}\widehat A&0&0\\0&\widehat S&0\\0&0&\widehat X\end{pmatrix}.
\end{equation}
Then finding the eigenvalues of $\mathcal{A}\bar{\mathcal{Q}}^{-1}$ is equivalent to solving
$$\bar{\mathcal{D}}^{-\frac{1}{2}}\mathcal{A}\bar{\mathcal{D}}^{-\frac{1}{2}}\mathbf{w}=\lambda \bar{\mathcal{D}}^{-\frac{1}{2}}\bar{\mathcal{Q}}\bar{\mathcal{D}}^{-\frac{1}{2}}\bold{w},$$
or 
\begin{equation}\label{Eq31}
\begin{pmatrix}\tilde{A}&R^T&0\\R&0&K^T\\0&K&0\end{pmatrix}\begin{pmatrix}x\\y\\z  \end{pmatrix}=\lambda \begin{pmatrix}I&R^T&0\\0&-I&K^T\\0&0&I\end{pmatrix}
\begin{pmatrix}x\\y\\z  \end{pmatrix},
\end{equation}
where $\tilde{A}=\widehat A^{-\frac{1}{2}}A\widehat A^{-\frac{1}{2}},\: R =\widehat S^{-\frac{1}{2}}B\widehat A^{-\frac{1}{2}}$
and $K = \widehat X^{-\frac{1}{2}}C\widehat S^{-\frac{1}{2}}$. 

\begin{Theorem}\label{Th6}
Suppose that $A \in \mathbb{R}^{n \times n}$ is symmetric positive definite and $B \in \mathbb{R}^{m \times n}$ and $C \in \mathbb{R}^{l \times m}$ are matrices with full row rank. Let $\widehat A, \widehat S$ and $\widehat X$ be the symmetric positive approximations of $A, S$ and $X$, respectively. 
Assume that $\lambda$ is an eigenvalue of the preconditioned matrix $\mathcal{A}\bar{\mathcal{Q}}^{-1}$ and $(x; y; z)$ is the corresponding eigenvector.
If $\Im(\lambda) \neq 0$, then $\lambda$ satisfies  
\begin{eqnarray*}
	|\lambda-1| &<  &\sqrt{1-\gamma^A_{\min}},  \qquad  \qquad \text{if} \ Ky = 0, \\
	|\lambda-1|  &\le & \sqrt{ 1 -\gamma^A_{\min} \frac{\|x\|^2}{\|y\|^2}}, \qquad  \text{if} \ Ky \ne 0. 
\end{eqnarray*}
\end{Theorem}

\begin{proof}
Let $\lambda$ be an eigenvalue of matrix $\mathcal{A}\bar{\mathcal{Q}}^{-1}$ and $(x; y; z)$ be the corresponding eigenvector such that $\|x\|^2+\|y\|^2+\|z\|^2=1$.
It follows from \eqref{Eq31} that 
\begin{align}
&\tilde{A}x-\lambda x=(\lambda -1) R^Ty, \label{Eq32}\\
&Rx-(\lambda -1)K^Tz=-\lambda y, \label{Eq33}\\
&Ky=\lambda z. \label{Eq34}
\end{align}
Since $\widehat A, \widehat S$ and $\widehat X$ are SPD, and $B$ and $C$ are matrices with full row rank, then $\tilde{A}$ is SPD and $K$ and $R$ are matrices with full row rank. 
If $y=0$, from \eqref{Eq32} we can derive $\tilde{A}x=\lambda x$. Thus, it can be deduced that $\lambda \in [\gamma^A_{\min},\gamma^A_{\max}]$ and then the  eigenvalue $\lambda$ is real. The associated eigenvector for this case is of the form $(x; 0; 0)$ where $x \ne 0$.
Assume now that $\lambda \ne 1$ and $y \ne 0$. The rest of the proof is divided into two cases:

\textbf{Case I.} $Ky=0$. From \eqref{Eq34}, we obtain $z=0$.
Multiplying \eqref{Eq32} by $x^*$ on the left and the transposed conjugate of \eqref{Eq33} by $y$ on the right, we get
\begin{align}
&x^*\tilde{A}x-\lambda \|x\|^2 =(\lambda -1)x^* R^Ty, \label{Eq35}\\
&x^*R^Ty=-\bar{\lambda} \|y\|^2. \label{Eq36}
\end{align}
Inserting \eqref{Eq36} into equation \eqref{Eq35}, we have
\begin{equation}\label{Eq37}
x^*\tilde{A}x-\lambda+(\lambda-\bar{\lambda})\|y\|^2+|\lambda|^2\|y\|^2 =0.
\end{equation}
Let $\lambda = a + ib$, then taking the real and imaginary parts of \eqref{Eq37} apart, we obtain
\begin{align}
&x^*\tilde{A}x-a+(a^2+b^2)\|y\|^2=0, \label{Eq38}\\
&b(2\|y\|^2-1)=0. \label{Eq39}
\end{align}
From \eqref{Eq39}, we have $b=0$ or $\|y\|^2=\frac{1}{2}$. We assume that $b\neq 0$. From \eqref{Eq37} and after some simple calculations, we have
\begin{equation}\label{Eq40}
2x^*\tilde{A}x-\lambda -\bar{\lambda}+|\lambda|^2=0.
\end{equation}\label{Eq41}
Using identity $|\lambda|^2-\lambda-\bar{\lambda}=|\lambda-1|^2-1$, we obtain $|\lambda-1|^2=1-2x^*\tilde{A}x$. If $1-\gamma^A_{\min} \geq 0$, we deduce that
\begin{equation*}\label{Eq42}
|\lambda-1|^2 \leq 1-\gamma^A_{\min},
\end{equation*}
which implies that 
\begin{equation*}\label{Eq43}
1-\sqrt{1-\gamma^A_{\min}} \leq \Re(\lambda) \leq 1+\sqrt{1-\gamma^A_{\min}}.
\end{equation*}
If $1-\gamma^A_{\min} < 0$, therefore there exists no $\lambda$ with nonzero imaginary part satisfying the equality in \eqref{Eq40}.

\textbf{Case II.} $Ky \neq 0$. Multiplying \eqref{Eq32} by $x^*$ on the left, the transposed conjugate of \eqref{Eq33} by $y$ on the right and \eqref{Eq34} by $z^*$ on the left, we derive
\begin{align}
&x^*\tilde{A}x-\lambda \|x\|^2=(\lambda -1)x^* R^Ty, \label{Eq44}\\
&x^*R^Ty-(\bar{\lambda}-1)z^*Ky=-\bar{\lambda} \|y\|^2, \label{Eq45}\\
&z^*Ky=\lambda \|z\|^2.\label{Eq46}
\end{align}
Inserting \eqref{Eq46} and \eqref{Eq45} into equation \eqref{Eq44} and easy manipulations, we get
\begin{equation}\label{Eq47}
x^*\tilde{A}x-\lambda(\|x\|^2+\|z\|^2)+(\lambda^2+|\lambda|^2-\lambda |\lambda|^2) \|z\|^2+\bar{\lambda}(\lambda-1)\|y\|^2=0.
\end{equation}
Using identity $\|x\|^2+\|z\|^2=1-\|y\|^2$, the above expression becomes
\begin{equation}\label{Eq48}
x^*\tilde{A}x-\lambda+(\lambda-\bar{\lambda}+|\lambda|^2)\|y\|^2+(\lambda^2-\lambda |\lambda|^2+|\lambda|^2)\|z\|^2=0,
\end{equation}
and can be equivalently written as
\begin{equation} 
x^* \tilde A x  - \lambda + \left(\lambda - \bar \lambda + |\lambda|^2\right) \|y\|^2 - 
\lambda \left(|\lambda|^2 - (\lambda + \bar \lambda)\right) \|z\|^2  = 0.
\end{equation}
In case of complex eigenvalues, we will show that the  real quantity
\begin{equation}\label{Eq50}
\rho = |\lambda|^2 - (\lambda + \bar \lambda) = |\lambda - 1|^2 -1 = |\lambda|^2 - 2a,
\end{equation}
is always negative, showing that the complex eigenvalues lie in an open circle with center $(1,0)$ and prescribed radius.
Let us write \eqref{Eq48}, exploiting the real and imaginary part,
\begin{align}
x^* \tilde A x -a + |\lambda|^2 \|y\|^2 - a \rho \|z\|^2&=0, \label{real}\\
b(-1 + 2 \|y\|^2 - \rho \|z\|^2)&=0.\label{imag}
\end{align}
If $\lambda$ is complex, then $b \ne 0$ and from (\ref{imag})  we obtain
\begin{equation} 
\|z\|^2 = \frac{2\|y\|^2 -1}{\rho}, \label{imag01} 
\end{equation}
and substituting it in (\ref{real}) we have
\[ 0 = x^* \tilde A x - a + |\lambda|^2 \|y\|^2 - a (2\|y\|^2 -1) = x^* \tilde A x + \|y\|^2(|\lambda|^2 - 2a) = x^* \tilde A x + \|y\|^2 \rho,\]
from which
\begin{equation}\label{Eq51}
\rho = -\frac{x^* \tilde A x}{\|y\|^2}.
\end{equation}
We can rewrite \eqref{Eq51} as
 \[ \rho = - \gamma_A \frac{\|x\|^2}{\|y\|^2} \le - \gamma^A_{\min} \frac{\|x\|^2}{\|y\|^2}, \qquad \gamma_A=\frac{x^* \tilde A x}{\|x\|^2}. \]
which together with \eqref{Eq50} completes the proof of the theorem.
\end{proof}

In the following, our aim is to characterize the real eigenvalues of the preconditioned matrix
not lying in $[\gamma^A_{\min},\gamma^A_{\max}]$. To this end, we premise two technical lemmas which will be useful for our analysis.
\begin{Lemma}(\cite{Bergamaschi})\label{Le1}
Let $\lambda \notin  [\gamma^A_{\min},\gamma^A_{\max}]$. Then for arbitrary $z \neq 0$, there exists a vector $s \neq 0$ such that 
\begin{equation*}
\frac{z^T(\tilde{A}-\lambda I)^{-1}z}{z^Tz}=\Big(\frac{s^T\tilde{A}s}{s^Ts}-\lambda \Big)^{-1}=(\gamma_A-\lambda)^{-1},
\end{equation*}
where $\gamma_A=\dfrac{s^T\tilde{A}s}{s^Ts}$.
\end{Lemma}

\begin{Lemma}\label{Le2}
Let $p(x)$ be the polynomial defined as
\[p(x)=x^3-a_1x^2+a_2x-a_3, \qquad a_j >0, \ j=1, 2, 3,\]
and let $a = \min\{a_1,\dfrac{a_3}{a_2}\}$ and $b = \max\{a_1,\dfrac{a_3}{a_2}\}$. Then $p(x) < 0, \: \forall x \in (0,a)$ and $p(x) > 0, \: \forall  x > b.$
\end{Lemma}

\begin{proof}
The statement of the lemma comes from observing that $p(x)$ is the sum of the term $x^3 - a_1x^2$
which is negative in $(0, a_1)$ and positive for $x > a_1$ and of the term $a_2x - a_3$ which is increasing and changes sign once for $x = \dfrac{a_3}{a_2}.$
\end{proof}

Let, as in the previous lemma,	$\gamma_A=\dfrac{s^T\tilde{A}s}{s^Ts}$ and define 
$\gamma_{S}=\dfrac{y^TRR^Ty}{y^Ty} = \dfrac{y^T\widehat S^{-1/2} \tilde{S} \widehat S^{-1/2}y}{y^T y}$, hence $\gamma_{S} \in  [\gamma^{S}_{\min}, \gamma^{S}_{\max}]$
and $\gamma_{X}=\dfrac{z^TKK^Tz}{z^Tz}=\dfrac{z^T\widehat S^{-1/2} \tilde{X} \widehat S^{-1/2}z}{z^T z}  \in  [\gamma^X_{\min}, \gamma^X_{\max}]$.
We are now able to bound the real eigenvalues of the preconditioned matrix $\mathcal{A}\bar{\mathcal{Q}}^{-1}$. We split the main
theorem considering two cases $K y = 0$ and $K y \ne 0$.

\begin{Theorem}
\label{Theorem_complex}
If $Ky = 0$, then the
real 	eigenvalues of the preconditioned matrix  not lying in 
$[\gamma^A_{\min},\gamma^A_{\max}]$ satisfy
\begin {eqnarray*}  \lambda^2-(\gamma_A+\gamma_{S})\lambda+\gamma_{S}  =  0.
\end{eqnarray*}
Moreover the following synthetic bound holds:
\begin{equation}  
	\label{boundsKy0}
	\min\left\{ \gamma^A_{\min}, \frac{\gamma^S_{\min}}{\gamma^A_{\max}+\gamma^S_{\min}} \right\}  \leq  \lambda \leq \gamma^A_{\max}+\gamma^S_{\max}.
\end{equation}  
\end{Theorem}

\begin{proof}
From \eqref{Eq32}, we have 
\begin{equation}\label{Eq54}
x =(\tilde{A}-\lambda I)^{-1}(\lambda -1)R^Ty.
\end{equation}
Inserting $x$ into the equation \eqref{Eq33} yields
\begin{equation}\label{Eq55}
R(\tilde{A}-\lambda I)^{-1}(1-\lambda)R^Ty=\lambda y.
\end{equation}
Multiplying the above equation by $\dfrac{y^T}{y^Ty}$ and using Lemma \ref{Le1}, we derive
\begin{equation}\label{Eq56}
\lambda^2-(\gamma_A+\gamma_{S})\lambda+\gamma_{S}=0,
\end{equation}
The solutions of equation \eqref{Eq56} are
\begin{equation*}\label{Eq57}
\lambda_{1,2}=\frac{\gamma_A+\gamma_{S}\pm \sqrt{(\gamma_A+\gamma_{S})^2-4\gamma_{S}}}{2}.
\end{equation*}
It is easy to see that 
\begin{equation*}\label{Eq58}
\lambda_{1}=\frac{\gamma_A+\gamma_{S}+ \sqrt{(\gamma_A+\gamma_{S})^2-4\gamma_{S}}}{2} \leq \gamma_A+\gamma_{S} \leq \gamma_{\max}^A+\gamma^S_{\max}.
\end{equation*}
It is not hard to find that the smallest eigenvalue $\lambda_2$ is a decreasing function with respect to $\gamma_A$ and it is an increasing  with respect to $\gamma_{S}$ if $\gamma_A \geq 1$. Therefore,  we have 
\begin{align*}\label{Eq59}
\lambda_{2}&=\frac{\gamma_A+\gamma_{S}- \sqrt{(\gamma_A+\gamma_{S})^2-4\gamma_{S}}}{2}\nonumber \\
&=\frac{2\gamma_{S}}{\gamma_A+\gamma_{S}+ \sqrt{(\gamma_A+\gamma_{S})^2-4\gamma_{S}}}
\geq \frac{\gamma^S_{\min}}{\gamma^A_{\max}+\gamma^S_{\min}}.\nonumber
\end{align*}
From the above discussion, we have proved that the real eigenvalues satisfy 
\begin{equation}\label{Eq58-0}
\frac{\gamma^S_{\min}}{\gamma^A_{\max}+\gamma^S_{\min}} \leq \lambda \leq \gamma^A_{\max}+\gamma^S_{\max}.
\end{equation}
\end{proof}

Before developing bound on the real eigenvalues of the preconditioned matrix in the general case we state the following  Lemma.

\begin{Lemma}\label{Le4}
Let $\zeta \in \mathbb{R}$ be either  $0 < \zeta < \min\left\{\gamma_{\min}^A, \dfrac{\gamma_{\min}^S}{\gamma_{\max}^A+\gamma_{\min}^S}\right\}$ or $\zeta \ge \gamma_{\max}^A+\gamma_{\max}^S$.
Then the symmetric matrix
\begin{equation}\label{Eq622}
Z(\zeta) = (1-\zeta )  R (\zeta I - \tilde A)^{-1} R^T + \zeta  I,
\end{equation}
has either all positive  or all negative eigenvalues.
\end{Lemma}

\begin{proof}
Let $w$ be a nonzero vector. Multiplying \eqref{Eq622} by
$\dfrac{w^T}{w^Tw}$ on the left and by
$w$ on the right and applying Lemma \ref{Le1}, since
$\zeta I -  \tilde A$ has all positive or all negative eigenvalues, yields
\begin{eqnarray}\label{Eq63}
\frac{w^T Z w}{w^T w} &=& (1-\zeta) \frac{w^T R (\zeta I - \tilde A)^{-1} R^T w}{w^T w}+\zeta   =
\quad (\text{setting} \ z = R^T w) \nonumber \\
&=& (1-\zeta) \frac{z^T (\zeta I - \tilde A)^{-1} z}{z^T z} \frac{w^T R R^T w}{w^T w} + \zeta   = \nonumber \\
        & = &   \frac{1-\zeta}{\zeta - \gamma_A} \gamma_S + \zeta.
\end{eqnarray}
        The Rayleigh quotient associated to $Z(\zeta)$, namely the function $h(\zeta) = \dfrac{1-\zeta}{\zeta - \gamma_A} \gamma_S + \zeta$ can not be zero under the hypotheses on $\zeta$.
In fact,
\begin{equation}\label{Eq64}
h(\zeta) = 0 \Longrightarrow \zeta^2 - (\gamma_A+\gamma_S) \zeta +\gamma_S  = 0,
\end{equation}
        and applying (\ref{Eq58-0}) we obtain the desired result.
\end{proof}

The next theorem provides bounds on the real eigenvalues of the preconditioned matrix $\mathcal{A}\bar{\mathcal{Q}}^{-1}$ in the general case.

\begin{Theorem}
	\label{Theorem_real}
Let $\lambda \in \mathbb{R}$ and $\lambda \notin 
	\left[\min\left\{\gamma_{\min}^A, \dfrac{\gamma_{\min}^S}{\gamma_{\max}^A+\gamma_{\min}^S}\right\}, \gamma_{\max}^A+\gamma_{\max}^S\right ]$.
	Then the remaining real
	eigenvalues of the preconditioned matrix $\mathcal{A}\bar{\mathcal{Q}}^{-1}$
satisfy 
	\begin {equation}
		\label{CubicPoly}
\lambda^3-(\gamma_A+\gamma_{S}+\gamma_{X})\lambda^2+(\gamma_{S}+\gamma_{X}+
	\gamma_A\gamma_{X})\lambda-\gamma_A\gamma_{X}=  0.
\end{equation}
Moreover the following synthetic bound holds:
	\begin {equation}  
		\label{boundsKynot0}
	\min \left \{\dfrac{\gamma_{\min}^S}{\gamma_{\max}^A+\gamma_{\min}^S},
\frac{\gamma_{\min}^A \gamma_{\min}^X}{\gamma_{\min}^X+ \gamma^S_{\max} + \gamma_{\min}^A\gamma_{\min}^X}\right\} \le \lambda  \le
\gamma_{\max}^A+\gamma_{\max}^{S}+\gamma_{\max}^{X}.
\end{equation}
\end{Theorem}

\begin{proof}
	The equation  \eqref{Eq32} can be written as
\begin{equation}\label{Eq65}
x = (1-\lambda) (\lambda I -\tilde{A})^{-1} R^T y.
\end{equation}
When we insert this into the second equation in \eqref{Eq33}, we obtain
\begin{equation}\label{Eq66}
Z(\lambda)y = (\lambda-1) K^T z,
\end{equation}
where $Z(\lambda)=(1-\lambda) R(\lambda I -\tilde{A})^{-1} R^T +\lambda I$.
        The hypotheses on $\lambda$ allow to use
                Lemma \ref{Le4} which guarantee  the matrix  $Z(\lambda)$ is either SPD or symmetric negative definite. 
		Hence,  obtaining
$y = (\lambda - 1)Z(\lambda)^{-1} K^T z$ from the previous equation and substituting in \eqref{Eq34} yields
\begin{equation}\label{Eq67}
\left[K(\lambda-1)Z(\lambda)^{-1} K ^T - \lambda I \right]z = 0.
\end{equation}
Premultiplying by $z^T$on the left and dividing by $z^Tz$ yields
\begin{equation}\label{Eq68}
\frac{z^TK (\lambda-1) Z(\lambda)^{-1} K ^T z}{z^Tz}-\lambda=0.
\end{equation}
Setting $w=K^Tz$, we can obtain
\begin{equation}\label{Eq68}
	(\lambda-1) \frac{w^TZ(\lambda)^{-1}w}{w^Tw} \frac{z^T K K^T z}{z^T z}-\lambda =0.
\end{equation}
	Denoted the  vector $u = Z(\lambda)^{-1/2}w$, the equation (\ref{Eq68}) becomes
\begin{equation}\label{Eq68_1}
	(\lambda-1) \frac{u^T u}{u^T Z(\lambda) u} \frac{z^T K K^T z}{z^T z}-\lambda =0.
\end{equation}
	Using now the relation \eqref{Eq63} in Lemma \ref{Le4}, we get
\begin{equation}
  \frac{\lambda-1}{\frac{1-\lambda}{\lambda - \gamma_A} \gamma_S + \lambda} \gamma_X - \lambda = 0.
\end{equation}
After simple algebra we are left with the following polynomial cubic equation
\begin{equation}\label{Eq61}
q(\lambda) \equiv \lambda^3-(\gamma_A+\gamma_{S}+\gamma_{X})\lambda^2+(\gamma_{S}+\gamma_{X}+
\gamma_A\gamma_{X})\lambda-\gamma_A\gamma_{X}=0.
\end{equation}
	Applying Lemma \ref{Le2} to this cubic polynomial we have
	\[ a = \frac{\gamma_A  \gamma_{X}}{\gamma_{X}+ \gamma_S + \gamma_A \gamma_{X}}, \qquad
	b = \gamma_A+\gamma_{S}+\gamma_{X} .\]
	In this case it is easily verified that $a < b$ from which we have that
	$ a < \lambda < b$ and	the statement of the theorem 
	results by observing that the lower bound is an increasing function  of both $\gamma_A$ and  $\gamma_K$
	and decreasing on $\gamma_S$.
\end{proof}

%%%%%%%%%%%%%%%%%%%%%%%%%%%%%%%

\smallskip

\noindent
\textbf{Check of the bounds in Theorems \ref{Theorem_complex} and \ref{Theorem_real}}.
 Figure \ref{fig3} displays in depth the eigenvalue distribution of preconditioned matrix $\bar{\mathcal{Q}}^{-1}\mathcal{A}$. 
% In Example \ref{Ex1} the approximation used for matrix $A$ is selected as $\widehat A=LL^T$ where $L$ is computed using the incomplete 
 %Cholesky with tolerance $10^{-4}$.
%For Example \ref{Ex2} we take  $\widehat A=\text{diag}(A)$.  For both examples the approximations of  $S$ and $X$ are taken by $\widehat S=\text{tridiag}(B\widehat A^{-1}B^{T})$ and $\widehat X=\text{tridiag}(C\widehat S^{-1}C^T)$. The key observations for these figures can be summarized as follows:
%
%%%%%%%%%%%%%%%%%%%%%%%%%%%%%%%%
\begin{figure}[H]
    \centering
    \includegraphics[scale=0.6]{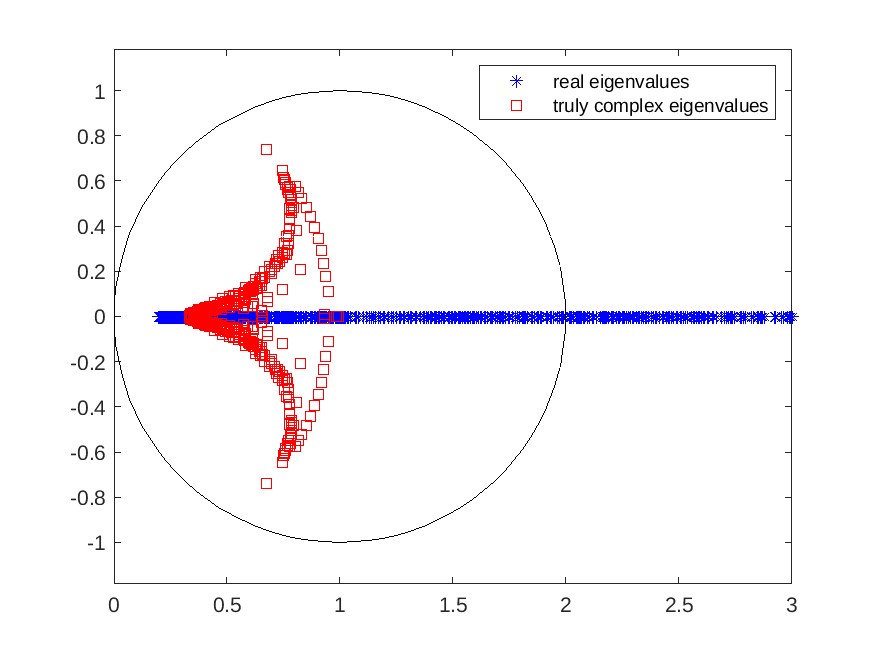}
    \caption{\small{Eigenvalue distribution of preconditioned matrix $\bar{\mathcal{Q}}^{-1}\mathcal{A}$ for Example  \ref{Ex2} with $p=16$.}}
    \label{fig3}
\end{figure}
\begin {itemize}
\item The complex eigenvalues of the preconditioned matrix $\bar{\mathcal{Q}}^{-1}\mathcal{A}$ fall in the open circle with center (1,0) and radius 1; 
\item Regarding the real eigenvalues, the results are summarized in the following table:

	\begin{tabular} {llll}
$\min\{\lambda, \lambda \in \R\} $&$\max\{\lambda, \lambda \in \R\} $ & Lower bound (\ref{boundsKynot0}) & Upper bound (\ref{boundsKynot0}) \\
\hline
%		1 & 0.0937 & 1.6705  & 0.0742 & 3.375 \\
		0.1982 & 3.0019 & 0.1342 & 6.2110
	\end{tabular}
\end {itemize}

%As a general comment we sometimes find that the lower bound is not so close to the smallest real eigenvalues, due to the approximation
%of the smallest real root of (\ref{Eq61}).

In the next section, we will perform a more accurate eigenvalues analysis of the preconditioned matrix with the $\bar{\mathcal{Q}}$ preconditioner, under additional hypotheses. 
%Subsequent section 5 will analyze more in depth the eigenvalue distribution 
%with this preconditioner.
%------------------------------------------------------
%%%%%%%%%%%%%%%%%%%%%%%%%%%%%%%%%%%%%%%%%%%%%%%
%-----------------------------------------------------
\section{Further characterization of real eigenvalues}\label{Fsec}
We will now consider a simplified preconditioner in which the only approximation is provided by $\widehat A \ne A$, whereas
$\widehat S=B\widehat A^{-1}B^T \equiv \tilde{S}$ and $\widehat X=C\widehat S^{-1}C^T \equiv \tilde{X}$. 
Note that $RR^T=I_m$ and $KK^T=I_l$. 

\begin{Theorem}\label{Th6}
	Let $\widehat S \equiv \tilde S$ and $\widehat X \equiv \tilde X$. Then
any real eigenvalue $\lambda$ of $\mathcal{A}\bar{\mathcal{Q}}^{-1}$ is bounded by
	\begin{equation*}
		\label{Statement_Th6}
 \min\left\{\lambda^+(\gamma^A_{\min}),\gamma^A_{\min}, \frac{1}{\gamma^A_{\max}+1}\right\} \le \lambda \le \max\{\lambda^+ (\gamma^A_{\max}), \gamma^A_{\max} +1\},
	\end{equation*}
where $\lambda^+(.)$ is the (unique) positive root of the equation
\[ \lambda^3 - (2 + \gamma_A) \lambda^2 + (2 + \gamma_A) \lambda - \gamma_A = 0.\]
Moreover,  the following more synthetic bound holds:
\begin{equation}
	\label{bound8}
\dfrac{\gamma^A_{\min}}{2} \le \lambda \le  \gamma^A_{\max}+1.
\end{equation}
\end{Theorem}

\begin{proof}
	For this simplified preconditioner we have $\gamma_S \equiv 1$ and $\gamma_X \equiv 1$.  In this case the equation (\ref{CubicPoly}) becomes
\begin{equation}\label{Eq71}
p(\lambda; \gamma_A) \equiv \lambda^3 - (2 + \gamma_A) \lambda^2 + (2 + \gamma_A) \lambda - \gamma_A = 0,
\end{equation}
for all real 
	\begin{equation}
		\label{conditions}
		\lambda \notin 
	\left[\min\left\{\gamma_{\min}^A, \dfrac{1}{1 + \gamma_{\max}^A}\right\}, 1+\gamma_{\max}^A\right ].
	\end{equation}
The cubic polynomial equation \eqref{Eq71} can be written as
\begin{equation}\label{Eq72}
p(\lambda; \gamma_A) \equiv \lambda \Big((\lambda-1)^2+1\Big) - \gamma_A (\lambda^2 - \lambda + 1) =  0,
\end{equation}
showing that the function $g(x) \equiv p(\lambda; x)$ is decreasing for each $x \ge 0$ and therefore
the position of the  largest 	positive root of \eqref{Eq72} is increasing.
Moreover it is easy to show that
for every $\gamma_A > 0$, there is a unique positive root to the equation $p(\lambda; \gamma_A)  = 0$. In fact
	\[ p(0; \gamma_A) = -\gamma_A < 0, \quad p'(0; \gamma_A) = 2 + \gamma_A, \qquad p'(\tilde \lambda; \gamma_A) = 0, 
		\quad \tilde  \lambda= \frac{\gamma_A+2 - \sqrt{\gamma_A^2+\gamma_A - 2}}{3},\]
	so that if $\gamma_A < 1$, the polynomial $p$ is increasing for $\lambda > 0$ and it takes a local maximum in $\tilde \lambda$ if $\gamma_A > 1$
		in which, however, $p(\tilde \lambda; \gamma_A) <  p(\tilde \lambda; 1) = 0$.
		Combining all these facts we finally have
		\[ \lambda^+(\gamma^A_{\min}) \le \lambda \le \lambda^+ (\gamma^A_{\max}),\]
		where $\lambda^+(\gamma_A)$ refers to the unique positive solution of $p(\lambda; \gamma_A) = 0$, and the thesis holds
		by observing that $p(\gamma^A_{\min}; \gamma^A_{\min}) > 0 \Longrightarrow \lambda^+(\gamma^A_{\min}) < \gamma^A_{\min}$ and
		$p(\gamma^A_{\max}; \gamma^A_{\max}) < 0 \Longrightarrow \lambda^+(\gamma^A_{\max}) > \gamma^A_{\max}$.

\noindent
		Also the second part of the theorem holds
		since $p\left(\dfrac{\gamma^A_{\min}}{2}; \gamma^A_{\min}\right) = -\dfrac{{(\gamma^A_{\min})}^3}{8} < 0$, then
		$\lambda^+(\gamma^A_{\min}) > \dfrac{\gamma^A_{\min}}{2}$. Moreover, from 
		$p\left(\gamma^A_{\max}+1; \gamma^A_{\max}\right) = 1 > 0$, we have       
		that $\lambda^+(\gamma^A_{\max}) < \gamma^A_{\max}+1$. 
		Combining this with (\ref{conditions}) and observing that $\dfrac{\gamma^A_{\min}}{2} < \dfrac{1}{\gamma^A_{\max} + 1}$,
		we conclude the proof.

\end{proof}

\begin{Remark}
	Note that we could have applied directly Lemma \ref{Le2} to equation (\ref{Eq71}),
	obtaining the following bounds
	\[ \frac{\gamma_{\min}^A}{2 + \gamma_{\min}^A} \le \lambda \le \gamma_{\max}^A + 2,\]
	which are looser than those proved in Theorem \ref{Th6}.
\end{Remark}

%%%%%%%%%%%%%%%%%%%%%%%%%%%%%%%%%%%%%
%%%%%%%%%%%%%%%%%%%%%%%%%%%%%%%%%%%%%%%%
\smallskip

\noindent
\textbf{Check of the bounds in Theorem \ref{Th6}}.
The following example is given to assess the theoretical results developed in Theorem \ref{Th6}.
\begin{Example}\label{Ex3}
Consider the linear system \eqref{Eq1} with the block matrices are randomly generated by the following MATLAB code:
\begin{lstlisting}
n = 100; m = 80; l = 60;
z = 1+10*rand(1); w = z*rand(n,1);  w = 0.1+sort(w);  w(1:10) = w(1);
A = diag(w); B = rand(m,n); C = rand(l,m);
\end{lstlisting}
\end{Example}

\noindent
In this example, matrix $A$ is diagonal with a random eigenvalue distribution in $[0.1, 11]$ and $\widehat A = I$. 
\begin{figure}[H]
    \centering
    \includegraphics[scale=0.5]{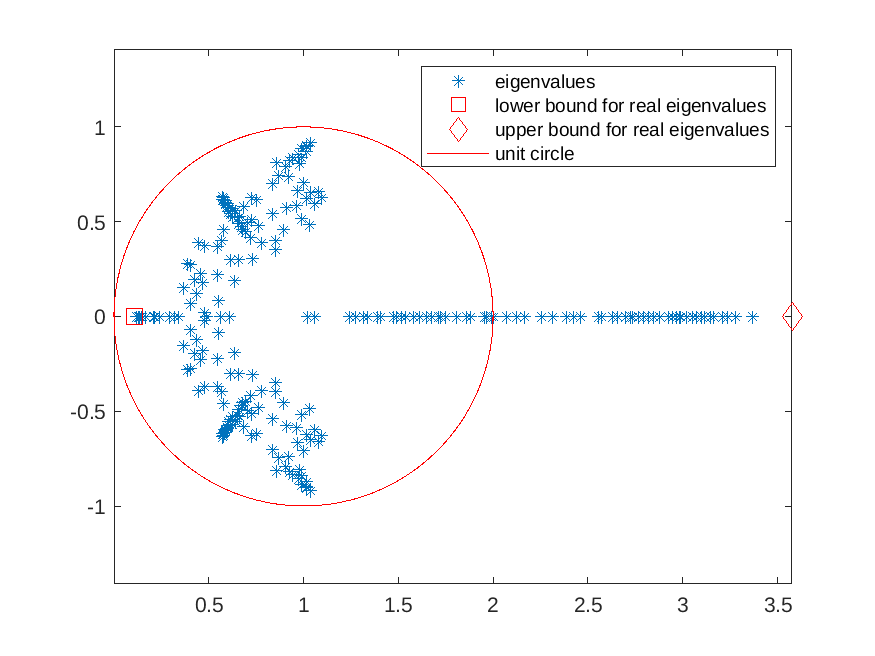}
	\hspace{-9mm}
    \includegraphics[scale=0.5]{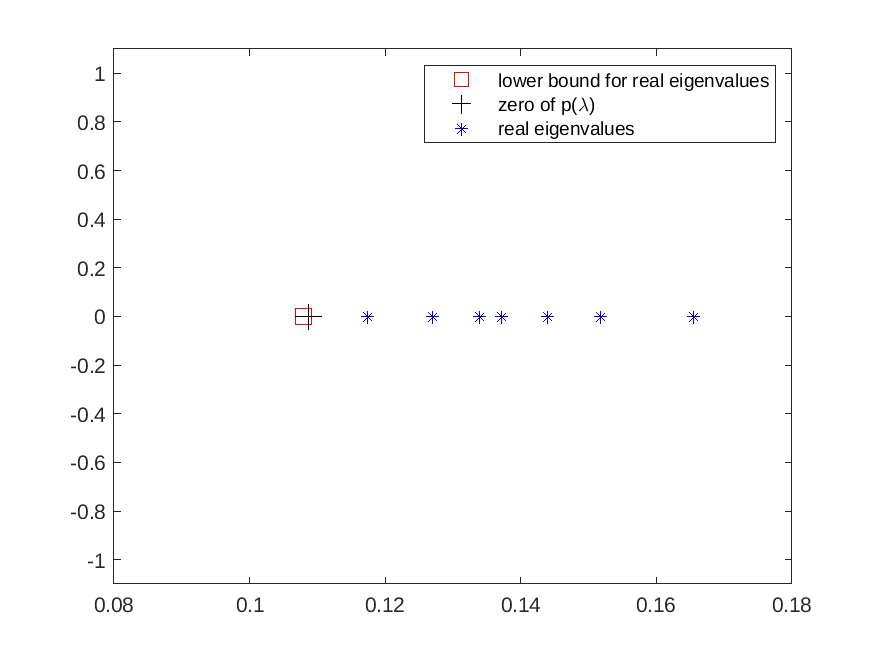}
	\caption{\small{Eigenvalue distribution of the preconditioned matrix $\bar{\mathcal{Q}}^{-1}\mathcal{A}$ for Example \ref{Ex3}. Left panel: real eigenvalues -- blue asterisks, complex eigenvalues -- red circles, upper bound by} (\ref{bound8}) -- red diamond. Right panel: smallest eigenvalues compared with 
	the two lower bounds of Theorem  \ref{Th6}.}
    \label{fig4}
\end{figure}

\noindent
In Figure \ref{fig4} (left) we show the whole spectrum of the preconditioned matrix $\bar{\mathcal{Q}}^{-1}\mathcal{A}$ together with the bounds for the real eigenvalues. In Figure  \ref{fig4} (right) 
a zoom of the smallest (real) eigenvalues is provided showing that both the lower bounds, namely
$\gamma_{\min}^A$ (red box) and $\lambda^+$ (black plus)  are smaller, yet very close, than the smallest real eigenvalue
of the preconditioned matrix. The results of this experiment as well as the observation of the figures
point out that:
\begin {itemize}
\item The complex eigenvalues of the related preconditioned matrix are located in a circle centered at (1, 0) with radius 1; 
\item The real eigenvalues lie in the real interval $[0.1024, 3.373]$; 
\item Here $\dfrac{\gamma^A_{\min}}{2}  = 0.1003, \ \lambda^{+} = 0.1008$  and   $\gamma^A_{\max}+1 = 3.552$
\end {itemize}
We can appreciate the closeness of the bounds to the endpoints of the real eigenvalue interval.
%s shown in the numerical results, in many cases  the real eigenvalues are located in the interval
%$[\gamma_{\min}^A, \gamma_{\max}^A]$.  When this does not occur, we have found very tight bounds 
%for all outliers. In any case, we can conclude that all real eigenvalues of the preconditioned matrix
%are strictly related to the interval $[\gamma_{\min}^A, \gamma_{\max}^A]$. Therefore, a good preconditioner ($\widehat A$) for $A$ will provide a favorable spectral distribution of the real eigenvalues of 
     %$\bar{\mathcal{Q}}^{-1}\mathcal{A}$.
%%-----------------------------------------------------
%%%%%%%%%%%%%%%%%%%%%%%%%%%%%%%%%%%%%%%%%%%%%%
%-----------------------------------------------------
\section{Conclusions}\label{Esec}
In this work, we have considered a number of \textit{exact} block preconditioners, developing
the spectral distribution of the corresponding preconditioned matrices, for a class of double saddle point problems. 
Some numerical experiments are performed, which  show the good behavior of the preconditioned FGMRES method using an  inexact counterpart of these
preconditioner, in comparison with other preconditioners from the literature.

We have then concentrated on the inexact variants of a specific block triangular preconditioner, performing a complete
spectral analysis and relating the eigenvalue distribution of the preconditioned matrix with the extremal eigenvalues
of the (symmetric and positive definite) preconditioned (1,1) block and the Schur complement matrices.
Numerical tests are reported which confirm the validity of the developed theoretical bounds. 

Future work is aimed at generalizing this work to provide the eigenvalue distribution of 
more general double saddle-point matrices, in particular those with nonzero
$(2,2)$ and $(3,3)$ blocks, and to test them on a wide number of realistic applications, such as, e.g.,  coupled poromechanical models  \cite{Frigo-et-al},
and the coupled Stokes-Darcy  equation \cite{BeikBenzi2022}.

%-----------------------------------------------------


\begin{thebibliography}{10}
\bibitem{Yuan}
J.-Y. Yuan, Numerical methods for generalized least squares problems, J. Comput. Appl. Math., 66 (1996), pp. 571--584.

\bibitem{Han}
D.-R. Han, X.-M. Yuan, Local linear convergence of the alternating direction method of multipliers for quadratic programs, SIAM J. Numer. Anal., 51 (2013), pp. 3446--3457.

\bibitem{Rhebergen}
S. Rhebergen, G.N. Wells, A.J. Wathen, R.F. Katz, Three-field block preconditioners for models of coupled magma/mantle dynamics, SIAM J. Sci. Comput., 37 (2015), pp. A2270--A2294.

\bibitem{Chen}
Z.-M. Chen, Q. Du, J. Zou, Finite element methods with matching and nonmatching meshes for Maxwell equations with discontinuous coefficients, SIAM J. Numer. Anal., 37 (2000), pp. 154--1570.

\bibitem{Monk}
P. Monk, Analysis of a finite element method for Maxwell's equations, SIAM J. Numer. Anal., 29 (1992), pp. 714--729.

\bibitem{Cai}
M. Cai, M. Mu,  J. Xu, Preconditioning techniques for a mixed Stokes/Darcy model in porous media applications, Comput. Appl. Math., 233 (2009), pp. 346--355.

\bibitem {ChenRen}
	F. Chen, B. Ren, On preconditioning of double saddle point linear systems arising from liquid crystal director modeling, Appl. Math. Lett, 136 (2023), 108445.

\bibitem{Szyld}
	 P. Chidyagwai, S. Ladenheim, D. B. Szyld,
     {Constraint preconditioning for the coupled {S}tokes-{D}arcy
              system},
		SIAM J. Sci. Comput., 38, (2016),
     pp. {A668---A690}.

	\bibitem{Benzi2018} F.P.A. Beik, M. Benzi, 
     Iterative methods for double saddle point systems,
		SIAM J. Matrix Anal. Appl., 39 (2018), pp. 
     {902--921}.

	\bibitem{BeikBenzi2022} F.P.A. Beik, M. Benzi, Preconditioning techniques for the coupled Stokes–Darcy problem: spectral and field-of-values analysis. Numer. Math. 150, 257--298 (2022). 

\bibitem{Cao}
Y. Cao, Shift-splitting preconditioners for a class of block three-by-three saddle point problems, Appl. Math. Lett., 96 (2019), pp. 40--46.

\bibitem{Bradley}
S. Bradley, C. Greif,
\newblock {Eigenvalue bounds for double saddle-point systems}.
		\newblock {IMA Journal of Numerical Analysis}, (2023). Published online on 23 December 2022.

\bibitem{Simoncini1}
V. Simoncini, D. Szyld, Recent computational developments in Krylov subspace methods for linear systems, Numer. Linear Algebra Appl., 14 (2007), pp. 1--59.

\bibitem{Huang1}
N. Huang, C.-F. Ma, Spectral analysis of the preconditioned system for the $3 \times 3$ block saddle point problem, Numer. Algor., 81 (2019), pp. 421--444.

\bibitem{Balani-et-al}
F. Balani Bakrani, M. Hajarian, L. Bergamaschi, Two block preconditioners for a class of double saddle point linear systems,
		Applied Numerical Mathematics, 190 (2023), pp. 155--167.

\bibitem{Xie}
X. Xie, H.B. Li, A note on preconditioning for the $3 \times 3$ block saddle point problem, Comput. Math. Appl., 79 (2020), pp. 3289--3296.

\bibitem{Wang}
N. N. Wang, J.-C. Li, On parameterized block symmetric positive definite preconditioners for a class of block three-by-three saddle point problems, Comput. Appl. Math., 405 (2022), 113959.

\bibitem{Simoncini}
V. Simoncini, Block triangular preconditioners for symmetric saddle-point problems, Appl. Numer. Math., 49 (2004), pp. 63--80.

\bibitem{AslSalBei}
H. Aslani, D.K. Salkuyeh, F.P.A. Beik,
On the Preconditioning of Three-by-Three Block Saddle Point Problems,
		{Filomat}, 35 (2021),
pp. {5181--5194}.

\bibitem{Bergamaschi}
L. Bergamaschi, On eigenvalue distribution of constraint-preconditioned symmetric saddle point matrices, Numer. Linear Algebra Appl., 19 (2012), pp. 754--772.

		\bibitem{Frigo-et-al}
		M. Frigo, N. Castelletto, M. Ferronato,
Enhanced relaxed physical factorization preconditioner for coupled poromechanics,
Comput. Math. Appl.,
106 (2022), pp. 27--39.
\end{thebibliography}
\end{document}